\documentclass[11pt]{article}

\usepackage{mathrsfs,amssymb,amsmath,amsfonts,dsfont}
\usepackage{amsthm}
\usepackage{bbm}
\usepackage{tipa}
\usepackage{enumerate}
\usepackage{color}
\usepackage[titletoc,title]{appendix}

\usepackage{cite}
\usepackage[hyphens]{url}
\usepackage[bookmarks=false,hidelinks]{hyperref}

\usepackage{arydshln}
\usepackage{mathtools}
\mathtoolsset{showonlyrefs}

\newtheorem{theorem}{Theorem}[section]
\newtheorem{lemma}[theorem]{Lemma}

\newtheorem{corollary}[theorem]{Corollary}

\theoremstyle{definition}

\newenvironment{keywords}{{\bf Key words: }}{}
\newenvironment{AMS}{{\bf AMS subject classification: }}{}

\numberwithin{equation}{section}
\newcommand{\overbar}[1]{\mkern 1.5mu\overline{\mkern-1.5mu#1\mkern-1.5mu}\mkern 1.5mu}

\newcommand{\abs}[1]{\left\vert#1\right\vert}
\newcommand{\norm}[1]{\left\Vert#1\right\Vert}

\newcommand{\udots}{\mathinner{\mskip1mu\raise1pt\vbox{\kern7pt\hbox{.}}
\mskip2mu\raise4pt\hbox{.}\mskip2mu\raise7pt\hbox{.}\mskip1mu}}

\title{{\bf Stochastic Differential Equations with Local Growth Singular Drifts}\footnote{Research supported by National Key R\&D Program of China (No. 2020YFA0712700) and the NSFC (No. 11931004, 12090014, 12288201) and the support by key Lab of Random Complex Structures and Data Science, Youth Innovation Promotion Association (2020003), Chinese Academy of Science. 
}
} 
\author{
{\bf  Wenjie Ye} 
\thanks{E-mail address: yewenjie@amss.ac.cn(W. Ye)}
\\
\footnotesize{$^{}$ Academy of Mathematics and Systems Science, Chinese Academy of Sciences, Beijing 100190, China}
}

\begin{document}
\maketitle
\begin{abstract}
In this paper, we study the weak differentiability of global strong solution of stochastic differential equations, the strong Feller property of the associated diffusion semigroups and {the global stochastic  flow property} in which the singular drift $b$ and the weak gradient of  Sobolev diffusion $\sigma$ are supposed to satisfy $\norm{\abs{b}\cdot\mathds{1}_{B(R)}}_{p_1}\le O((\log R)^{{(p_1-d)^2}/{2p^2_1}})$ and $\norm{\norm{\nabla \sigma}\cdot\mathds{1}_{B(R)}}_{p_1}\le O((\log ({R}/{3}))^{{(p_1-d)^2}/{2p^2_1}})$ respectively. The main tools for these results are the decomposition of global two-point motions in \cite{fang2007}, Krylov's estimate, Khasminskii's estimate, Zvonkin's transformation and the characterization for Sobolev differentiability of random fields in \cite{xie2016}.
\end{abstract}

\\


\section{Introduction and main results}
In this paper, we consider the following $d$-dimension stochastic differential equations (SDEs, for short) 
\begin{equation}\label{eq:b}
    \left\{\begin{aligned}
         & dX_t =  b(X_t)\,dt + \sigma(X_t)\,dW_t,\ t\in[0,T],\\
         & X_0=x\in \mathbb{R}^d.
\end{aligned}\right.
\end{equation}
Here, $\{W_t\}_{t\in[0,T]}$ is a standard Wiener process in $\mathbb{R}^d$ which defined on a complete filtered probability space $(\Omega,\mathscr{F},\mathbb{P},\{\mathscr{F}_t\}_{t\ge 0})$. The coefficients $b:\mathbb{R}^d\rightarrow \mathbb{R}^d$ and $\sigma:\mathbb{R}^d\rightarrow \mathbb{R}^{d\times d}$ are both  Borel measurable function. It is well-known that stochastic differential equation defined a global stochastic homeomorphism flow  if $b$ and $\sigma$ satisfy global Lipschitz conditions and linear growth conditions.
In the past decades, for the non-Lipschitz coefficients SDEs there is increasing interest about their solutions and their properties(for example, the strong completeness property, the weak differentiability,  stochastic homeomorphism flow property and so on).

Yamada and Ogura\cite{Yamada1981} proved the existence of global flow of homeomorphisms for one-dimensional SDEs under local Lipschitz and linear growth conditions. Li\cite{li1994} proved  the strong completeness property of  SDEs \eqref{eq:b} by studying the derivative flow equation of SDEs \eqref{eq:b}.  Fang and Zhang \cite{fang2005} used the Gronwall-type estimate to study SDEs under non(local) Lipschitz conditions. Fang, Imkeller and Zhang \cite{fang2007} proved Stratonovich equation defined
a global stochastic homeomorphism flow  if the coefficients are just locally Lipschitz and Lipschitz coefficients with mild growth. Chen and Li\cite{chen2014} studied Sobolev regularity of equation \eqref{eq:b} and strong completeness property when  $b$ and $\sigma$ are Sobolev coefficients.

When $\sigma=I$ and $b$ is bounded measurable, Veretennikov\cite{veretennikov1979} first proved existence and uniqueness of the strong solution. When $\sigma=I$ and $b$ satisfy
\begin{equation}\label{index-condition}
    \begin{aligned}
         & \left(\int_0^T \left(\int_{\mathbb{R}^d}\abs{b}^p\,dx\right)^{\frac{q}{p}}\,dt\right)^{\frac{1}{q}}<\infty,\quad p,q\in[2,\infty),\quad\frac{2}{q}+\frac{d}{p}<1,
    \end{aligned}
\end{equation}
Krylov and R\"ockner\cite{Krylov2005} using the technique of PDEs proved existence and uniqueness of the strong solution. The similar result in time-homogeneous case was obtained  by Zhang and Zhao \cite{zhang2018} and dropped the assumption $\int^t_0\abs{b(X_s)}^2\,ds<\infty,\  a.s. $. Fedrizzi and Flandoli\cite{fedrizzi2013} proved the existence of a stochastic flow of $\alpha$-Hölder homeomorphisms for solutions of SDEs  and weak differentiability of solutions of SDEs under condition \eqref{index-condition}. Zhang \cite{zhang2011b,zhang2016} extended the results of Krylov and R\"ockner\cite{Krylov2005} to the case of multiplicative noises,  the well posedness of solutions, the weak differentiability of solutions be obtained {and the solution forms a stochastic flow of homeomorphisms of $\mathbb{R}^d$ be proved}, the main tools are Krylovl's estimate and  Zvonkin's transformation.  In \cite{xie2016},  a characterization for Sobolev differentiability of random field be established. With the characterization,  the weak differentiability of solutions be proved  under local Sobolev integrability and sup-linear growth assumptions. We refer the reader for\cite{zvonkin1974,krylov1980,g2001,zhang2005,zhang2011b,zhang2016,xie2016,wang2016} and references therein  about the applications of Krylov's estimate, Zvonkin's transformation and the characterization for Sobolev differentiability of random field.  More recently, the critical case i.e. $p=d$ in time-homogeneous case, $\frac{2}{q}+\frac{d}{p}=1$ in time-inhomogeneous have been explored, see\cite{krylov2021a,krylov2021b,krylov2021c,krylov2021d,rockner2020,rockner2021} and references therein.

In \cite{fang2005}, Fang, Imkeller and Zhang obtained  a global estimates by using global decomposition of two-point motions and local estimates. In this paper, we will base on the decomposition,  Krylov's estimate, Khasminskii's estimate, Zvonkin's transformation and the characterization of Sobolev differentiability of random fields to obtain the well posedness and the weak differentiability of solutions, the strong Feller property of associated semigroups and {stochastic  flow property} of SDEs \eqref{eq:b} under the following assumption:
\begin{itemize}
    \item[$\mathbf{(H^b)}$] There exist two positive constants $\beta$ and $\tilde{\beta}$ such that for all $R\ge 1$,
\begin{equation}
    \begin{aligned}
          \left(\int_{B(R)} \abs{b(x)}^{p_1} \,dx\right)^{\frac{1}{p_1}}\le \beta I_b(R)+\tilde{\beta},
    \end{aligned}
\end{equation}
where $B(R):=\{x\in \mathbb{R}^d; \abs{x}\le R\}$ is a ball with  center $0$ and radius $R$, $\abs{\cdot}$ denote the Euclidiean norm, $p_1>d$ is a constant  and $I_b(R)=(\log R+1)^{{(p_1 -d)^2 }/{(2p^2_1})}$. \\
{\item[{ $(\bf{H}^\sigma_1)$}] There exist a constant $\delta\in(0,1)$  such that  for all $x,\xi \in \mathbb{R}^d$,
\begin{equation}
    \delta^{\frac{1}{2}}\abs{\xi}\le \abs{\sigma^\top(x)\xi}\le \delta^{-\frac{1}{2}}\abs{\xi},
\end{equation}
and there exists a constant $\varpi \in (0,1) $ such that  for all $x,y\in \mathbb{R}^d$,
$$
\norm{\sigma(x)-\sigma(y)}\le \delta^{-\frac{1}{2}}\abs{x-y}^\varpi.
$$
Here, we denote $\sigma^\top$ the transpose of matrix $\sigma$, $\norm{\cdot}$ the Hilbert-Schmidt norm.}

\item[{ $(\bf{H}_2^\sigma)$}]There exist two positive constants $\beta$ and $\tilde{\beta}$ (same with $\mathbf{(H^b)}$) such that for all $R\ge 1$,
\begin{equation}
    \left(\int_{B(R)}\norm{\nabla \sigma}^{p_1}\,dx\right)^{\frac{1}{p_1}}\le \beta I_\sigma(R)+\tilde{\beta}, 
\end{equation}
where $\nabla \sigma:=[\nabla \sigma^1,\cdots,\nabla\sigma^d]$ and $I_\sigma(R)=(\log({R}/{3})+1)^{{(p_1 -d)^2 }/{(2p^2_1})}$.
\end{itemize}

Our main results are given as the
following theorem:
{
\begin{theorem}\label{main-theorem}
   Under the conditions  $\mathbf{(H^b)}$, $\mathbf{(H^\sigma_1)}$ and $\mathbf{(H^\sigma_2)}$, there exists a unique global strong solution to \eqref{eq:b}. Moreover, we have the following conclusions:
\begin{enumerate}[(A)]
    \item For all $t\in [0,T]$ and almost all $\omega$, the mapping $x\mapsto X_t(\omega,x)$ is Sobolev differentiable and for any $p\ge 2$, there exist constants $\mathbf{C},n>0$ such that for Lebesgue almost all $x\in \mathbb{R}^d$,
    \begin{equation*}
        \begin{aligned}
             & \mathbb{E}\left[ \sup_{t\in[0,T]} \norm{\nabla X_t(x)}^p \right]\le \mathbf{C}(1+\abs{x}^n),
        \end{aligned}
    \end{equation*}
where $\nabla$ denotes the gradient in the distributional sense.
\item For any $t\in [0,T]$ and any bounded measurable function $f$ on $\mathbb{R}^d$, 
\begin{equation*}
    \begin{aligned}
         & x\mapsto \mathbb{E}[f(X_t(x))] \text{\ is continuous},
    \end{aligned}
\end{equation*}
i.e. the semigroup $P_t f(x):= \mathbb{E}[f(X_t(x))]$ is strong Feller.
\item {For all $t\in[0,T]$, $x\in \mathbb{R}^d$ and almost all $\omega$, the mapping $(t,x)\mapsto X_{t}(\omega,x)$ is continuous on $[0,T]\times\mathbb{R}^d$ and for almost all $\omega$,  $x\mapsto X_t(\omega,x)$ is one-to-one on $\mathbb{R}^d$. }
\end{enumerate}
\end{theorem}

These results will be proved in section \ref{sec6}.

We would like to compare the work in\cite{zhang2018,zhang2011b,xie2016} with the present paper and explain the contributions made in this paper. Following the proof of \cite{zhang2018}, we generalized  \cite[Theorem $3.1$]{zhang2018} to multiplicative noises (cf. Theorem \ref{th:XR-exist}). {In the time-inhomogeneous case, Xie and Zhang\cite{xie2016} proved the weak differentiability of SDEs and the strong Feller property of the associated diffusion semigroup under local Sobolev integrability and  sup-linear growth assumptions.  In the present paper, we removed the sup-linear growth condition (H2) in \cite{xie2016} by replacing the local Sobolev integrability (H1) in \cite{xie2016} with stronger assumptions $\mathbf{(H^b)}$, $\mathbf{(H^\sigma_1)}$ and $(\mathbf{H^\sigma_2})$, proved  the weak differentiability of SDEs and the strong Feller property of the associated diffusion semigroup in the time-homogeneous case.}  In the time-inhomogeneous case, Zhang \cite{zhang2011b} proved the solution of SDEs {forms a stochastic flow of homeomorphisms} under conditions: 
$$\abs{b},\ \norm{\nabla \sigma} \in L^{p_1}_{loc}(\mathbb{R}_+;L^{p_1}(\mathbb{R}^d))\ (p_1>d+2).$$ In the time-homogeneous case, the conditions will be
\begin{equation}\label{b-sigma-conditions}
   \abs{b},\ \norm{\nabla \sigma} \in L^{p_1}(\mathbb{R}^d)\ (p_1>d). 
\end{equation}
{Our main result Theorem \ref{main-theorem}(C) strengthen the   one-to-one property of stochastic flow  in \cite[Theorem $1.1$]{zhang2011b} by improving the conditions  \eqref{b-sigma-conditions} with mild growth conditions $(\mathbf{H}^b)$ and $(\mathbf{H}^\sigma_2)$.}

For the proof of Theorem \ref{main-theorem}, there are two main difficulties. The one is finer estimates depend on $R$ is necessary for us to obtain the order of growth in $\mathbf{(H^b)}$ and $\mathbf{(H^\sigma_2)}$ by the decomposition of global two-point motions. By our knowledge, all existing results about Krylov's estimate and Khasminskii's estimate such as \cite{zhang2018,zhang2011b,zhang2016,xie2016} do not obviously depend on radius $R$.  

Another difficulty is that we need an appropriate truncation for $\sigma$ due to SDEs \eqref{eq:b} with multiplicative noises. If we directly truncate $\sigma$ by characteristic function $\mathds{1}_{\abs{x}\le R}$, then the truncated $\sigma$ will be degenerate. Chen and Li\cite{chen2014} provides a truncation method which can guarantee truncated $\sigma$ is not degenerate, but it seems difficult to estimate the gradient of  truncated $\sigma$ by $(\bf{H}_2^\sigma)$.

We also give some remarks related to the proof of our main results and conditions posed in it.
\begin{itemize}
    \item In  Theorem \ref{main-theorem}, we just consider the time-homogeneous case, but  by carefully tracking the proof of Theorem \ref{main-theorem}, Our idea still work for time-inhomogeneous case.
    \item {If the condition $(\bf{H}^\sigma_1)$ of Theorem \ref{main-theorem}  be replaced by\\
   { $(\bf{H}^\sigma_1)_{loc}$} There exist a constant $\delta_R\in(0,1)$ depend on $R$ such that  for all $x\in B(R),\xi \in \mathbb{R}^d$,
\begin{equation}
    \delta^{\frac{1}{2}}_R\abs{\xi}\le \abs{\sigma^\top(x)\xi}\le \delta^{-\frac{1}{2}}_R\abs{\xi},
\end{equation}
and there exists two constants $L>0$ and $\varpi \in (0,1) $ such that  for all $x,y\in \mathbb{R}^d$,
$$
\norm{\sigma(x)-\sigma(y)}\le L\abs{x-y}^\varpi,
$$
where the growth of $\delta^{-1}_R$ be  mild  about $R$. The techniques in the proof of Theorem \ref{main-theorem} still can be used. Indeed, if $b$ and $\sigma$ satisfy
$\norm{\abs{b}\cdot\mathds{1}_{B(R)}}_{p_1}\le O(\tilde{I}_b(R))$, $\norm{\norm{\nabla \sigma}\cdot\mathds{1}_{B(R)}}_{p_1}\le O(\tilde{I}_b(R/3))$ and the assumption $(\bf{H}^\sigma_1)_{loc}$ holds true, then the following assumptions:\\
{$(\bf{H}^{\sigma^R}_1)_{loc}$} There exist a positive constant $\tilde{\delta}_R^{-{1}/{2}}=\mathbf{C}(d,L)\cdot(\delta_R^{-{1}/{2}})>0$ depend on $R$ such that for all $x,\xi\in \mathbb{R}^d$,
\begin{equation}
    \tilde{\delta}_R^{\frac{1}{2}}\abs{\xi}\le \abs{(\sigma^R)^\top(x)\xi} \le \tilde{\delta}_R^{-\frac{1}{2}}\abs{\xi},  
\end{equation}
and for all $x,y\in\mathbb{R}^d$,
\[
\norm{\sigma^R(x)-\sigma^R(y)}\le \tilde{\delta}^{-\frac{1}{2}}_R\abs{x-y}^\varpi.
\]\\
$(\bf{H}_2^{\sigma^R})_{loc}$ There exist  constants $\mathbf{C}(d,L)$ such that for all $R\ge 1$,
\begin{equation}
    \left(\int_{\mathbb{R}^d} \norm{\nabla \sigma^R}^{p_1}\,dx\right)^\frac{1}{p_1} \le \mathbf{C}(d,L)\cdot\tilde{\delta}^{-\frac{1}{2}}_{3R}+O(\tilde{I}_b(R)),
\end{equation}
hold true, where $O(\tilde{I}_b(R))$ means there exist two constants $C>0$ and $R_0$ such that $O(\tilde{I}_b(R))\le C\tilde{I}_b(R)\ \,\forall \,R\ge R_0$.   On the other hand, by going through carefully the proof of Theorem \ref{studied-pde-theorem} we can find two continuous increasing functions $G_1:\mathbb{R}_{+}\rightarrow \mathbb{R}_{+}$ and $G_2:\mathbb{R}_{+}\rightarrow \mathbb{R}_{+}$ such that $C_1$ and $C_2$ in Theorem \ref{studied-pde-theorem} are equal to $G_1(\tilde{\delta}_R^{-\frac{1}{2}})$ and $G_2(\tilde{\delta}_R^{-\frac{1}{2}})$. The $C_0(\tilde{\delta}_R^{-\frac{1}{2}})$ (the key to obtain $G_1$) in the proof of Theorem \ref{studied-pde-theorem} can be obtained by changing of coordinates to reduce $L^{\sigma^R(x_0)}$ to $\Delta$.
The $C_j(\tilde{\delta}_R^{-\frac{1}{2}})$ and $k_j(\tilde{\delta}_R^{-\frac{1}{2}})$ in \eqref{heat_kernel} (the key to obtain $G_2$) can be obtained by going through carefully the proof of Page $356$ to Page $378$ in \cite{ladyzenskaja1968}. Finally, we can take $\tilde{\delta}^{-\frac{1}{2}}_{3R}$ satisfy $\mathbf{C}(d,L)\cdot\tilde{\delta}^{-\frac{1}{2}}_{3R}\le \mathbf{C}\cdot\tilde{I}_b(R)$
and let $\lambda^R=(2G_2(\tilde{I}_b(R))\tilde{I}_b(R))^{2p_1/(p_1-d)}$ in Lemma \ref{lemma-u^R-estimate}. Tracking the proof in Theorem \ref{main-theorem}, we can find a concrete $\tilde{I}_b(R)$ with enough mild growth such that   the results in Theorem \ref{main-theorem} still hold true.
}
    \item In \cite{zhang2011b}, the well-known Bismut-Elworthy-Li's formula (cf. \cite{Elworthy1994}) was proved.   {But even if $\sigma(x)\equiv \sigma$(in this case, we do not need to truncate $\sigma$),} it seems difficult to prove the Bismut-Elworthy-Li's formula for the solution of SDEs \eqref{eq:b} under assumptions of this paper  due to $\mathbb{E}[\norm{\nabla X^R_t(x)}^2]\le C(R)$ and  $C(R)\rightarrow \infty$ when $R\rightarrow \infty$.
    \item {The local estimates \eqref{eq:X-x-y}, \eqref{t-s} and \eqref{1-plus-x^2} is seemingly not enough to obtain the onto property of the map $x\mapsto X_t(\omega,x)$. In fact, if we define 
$$\mathscr{X}_t(x):=\begin{cases}
    \left( 1+\abs{X_t(\frac{x}{\abs{x}^2})} \right)^{-1},\quad &x\neq 0,\\
0,\quad\quad &x=0.
\end{cases}$$ 
We just can obtain for any $k\in \mathbb{N}$, $x,y\in \{ x: \frac{1}{k}\le\abs{x}\le 1 \}\cup \{{0}\}$,
\begin{equation*}
    \begin{aligned}
         & \mathbb{E}\left[ \abs{\mathscr{X}_t(x)-\mathscr{X}_t(y)}^p \right]\le \mathbf{C}(k)\abs{x-y}^p.
    \end{aligned}
\end{equation*}
Notice that, the domain $\{ x: \frac{1}{k}\le\abs{x}\le 1 \}\cup \{{0}\}$ is not connected, we can not obtain $x\mapsto\mathscr{X}_t(x)$ exist a continuous version on $\{x:\abs{x}\le 1  \}$.
}
\item For the critical case i.e. $p_1=d$, our idea will not work since Zvonkin's transformation cannot be used. On the other hand, $\mathbf{(H^b)}$ and $\mathbf{(H^\sigma_2)}$ seemingly indicate the order of growth will be degenerate in the critical case.
\end{itemize}

The rest of this paper is organized as follows: In section \ref{sec2}, we will present some preliminary knowledge. In section \ref{sec3},  we devote to construct the cut-off functions to truncate SDEs \eqref{eq:b} and verify assumptions. In section \ref{sec4}, we provide a proof of Krylov's estimate and Khasminskii's estimate. In section \ref{sec5}, we use Zvonkin's transformation to estimate truncated SDEs \eqref{eq:newsde}. In section \ref{sec6}, we complete the proof of the main theorem \ref{main-theorem}. Finally, we give a detailed proof of Theorem \ref{studied-pde-theorem} in Appendix.

\section{Preliminary}\label{sec2}
In this section, we introduce some notations, function spaces and well-known theorems which will be used in this paper. 

We use $:=$ as a way of definition. Let $\mathbb{N} $ be the collection of all positive integer.   For any $a,b\in \mathbb{R}$, set $a\wedge b:=\min\{a,b\}$ and $a\vee b:=\max\{a,b\}$. We use $a\lesssim b$ to denote there is a constant $C$ such that $a\le Cb$, use
$a\asymp b$  to denote $a\lesssim b$ and $b\lesssim a$. For  functions $f$ and $g$, we use $f*g$ to denote the convolution of $f$ and $g$.

Let $L^p(\mathbb{R}^d)$ be  $L^p$-space on $\mathbb{R}^d$ with norm
\[
\norm{f}_p:= \left(\int_{\mathbb{R}^d}\abs{f}^p\,dx\right)^{\frac{1}{p}}<+\infty,\ \forall f\in L^p(\mathbb{R}^d).
\]
Let $W^{m,p}(\mathbb{R}^d)$ be  Sobolev space on $\mathbb{R}^d$ with norm
\[
\norm{f}_{m,p}:=\sum^m_{i=0}\norm{\nabla^i f}_p<+\infty,\ \forall f\in W^{m,p}(\mathbb{R}^d),
\]
where $\nabla^i$ denotes the $i$-order gradient operator. \\
For $0\le\alpha\in \mathbb{R}$ and $p\in [1,+\infty)$, the Bessel potential space $H^{\alpha,p}(\mathbb{R}^d)$ is defined by 
\[
H^{\alpha,p}:=(I-\Delta)^{-\frac{\alpha}{2}}(L^p(\mathbb{R}^d))
\]
with norm
\[
\norm{f}_{\alpha,p}:=\norm{(I-\Delta)^{\frac{\alpha}{2}} f}_p,\ \forall f\in H^{\alpha,p}(\mathbb{R}^d).
\]
Let $C^{\alpha}(\mathbb{R}^d)$ be H\"older space on $\mathbb{R}^d$ with norm
\[
\norm{f}_{C^\alpha}:=\sum^{\left\lfloor \alpha \right\rfloor  }_{i=0} \norm{\nabla^i f}_{\infty}+\sup_{x\neq y}\frac{\abs{\nabla^{\left\lfloor \alpha \right\rfloor  }f(x)-\nabla^{\left\lfloor \alpha \right\rfloor  }f(y)}}{\abs{x-y}^{\alpha-\left\lfloor \alpha \right\rfloor  }}<+\infty,\ \forall f\in C^{\alpha}(\mathbb{R}^d),
\]
where $\left\lfloor \alpha \right\rfloor  $ denotes the integer part of $\alpha$.
Let $C^\infty_0(\mathbb{R}^d)$ be a collection of all smooth function with compact support in $\mathbb{R}^d$.\\
For $\alpha\in (0,2)$ and $p\in (1,+\infty)$, we have 
\begin{equation}\label{frac-norm-equ}
  \norm{f}_{\alpha,p}\asymp \norm{(I-\Delta^{\frac{\alpha}{2}})f}\asymp \norm{f}_p +\norm{\Delta^{\frac{\alpha}{2}}f}_p,  
\end{equation}
where $\Delta^{\frac{\alpha}{2}}:=-(-\Delta)^{\frac{\alpha}{2}}$ is the fractional Laplacian.\\
Let $f$ be a locally integrable function on $\mathbb{R}^d$, $\mathcal{M}$ be the Hardy-Littlewood maximal operator defined by
\[
\mathcal{M}f(x):=\sup_{0<R<+\infty}\frac{1}{\abs{B(R)}}\int_{B(R)}f(x+y)\,dy,
\]
here, with a bit of abuse of notations, $\abs{B(R)}$ denotes the volume of ball $B(R)$.

\begin{theorem}[Sobolev embedding theorem]
If $k>l>0,p<d$ and $1\le p<q<\infty$ satisfy $k-\frac{d}{p}=l-\frac{d}{q}$, then
\[
H^{k,p}(\mathbb{R}^d)\hookrightarrow H^{l,q}(\mathbb{R}^d).
\]
If $\gamma\ge 0$ and $\gamma<\alpha-\frac{d}{p}$, then
\[
H^{\alpha,p}(\mathbb{R}^d)\hookrightarrow C^{\gamma}(\mathbb{R}^d).
\]

\end{theorem}

\begin{theorem}[Hadamard's theorem]\label{Hadamard's Theorem}
    If a function $\varphi:\mathbb{R}^d\rightarrow \mathbb{R}^d$ be a $k$-order smooth function ($k\ge 1$) and satisfy:
\begin{enumerate}[(i)]
    \item $\lim_{\abs{x}\rightarrow \infty}\abs{\varphi(x)}=\infty$;
    \item for all $x\in \mathbb{R}^d$, the Jacobian matrix $\nabla \varphi(x)$ is an isomorphism of $\mathbb{R}^d$;
\end{enumerate}
Then $\varphi$ is a $C^k$-diffeomorphism of $\mathbb{R}^d$.
\end{theorem}

\begin{theorem}\label{Max-inequality}
    \begin{enumerate}[(i)]
        \item There exist a constant $C_d$ such that for all $\varphi\in C^\infty(\mathbb{R}^d)$ and $x,y\in \mathbb{R}^d$, 
        \[
            \abs{\varphi(x)-\varphi(y)}\le C_d\cdot\abs{x-y}\cdot\bigl( \mathcal{M}\abs{\nabla \varphi}(x)+\mathcal{M}\abs{\nabla \varphi}(y) \bigr).
        \]
        \item For any $p>1$, there exist a constant $C_{d,p}$ such that for all $\varphi\in L^p(\mathbb{R}^d)$,
\[
\left(\int_{\mathbb{R}^d} \big( \mathcal{M}\varphi(x) \big)^p\,dx \right)^{\frac{1}{p}}\le C_{d,p} \left(\int_{\mathbb{R}^d} \abs{\varphi(x)}^p\,dx \right)^{\frac{1}{p}}.
\]
    \end{enumerate}
\end{theorem}

\section{Truncated SDEs}\label{sec3}
In this section, we will construct some precise cut-off functions to truncate SDEs \eqref{eq:b} and verify truncated SDEs
\begin{equation}\label{eq:newsde}
    \left\{\begin{aligned}
         & dX^R_t= b^R(X^R_t)\,dt + \sigma^R(X^R_t)d\widetilde{W}_t
        ,\  t\in [0,T],\\
          & X_0^R=x\in \mathbb{R}^d,
\end{aligned}\right.
\end{equation}
satisfy the following assumptions:
\begin{itemize}
    \item[$\mathbf{(H^{b^R})}$] There exist two positive constants $\beta$ and $\tilde{\beta}$ such that for all $R \ge 1$,
    \begin{equation}\label{eq:growth-inequality}
    \begin{aligned}
          \left(\int_{\mathbb{R}^d} \abs{b^R(x)}^{p_1} \,dx\right)^{\frac{1}{p_1}}\le \beta I_b(R)+\tilde{\beta},
    \end{aligned}
\end{equation}
where $p_1>d$ is a constant.
{\item[{ $(\bf{H}^{\sigma^R}_1)$}] There exist a positive constant $\tilde{\delta}\in(0,1)$ such that for all $x,\xi\in \mathbb{R}^d$,
\begin{equation}
    \tilde{\delta}^{\frac{1}{2}}\abs{\xi}\le \abs{(\sigma^R)^\top(x)\xi} \le \tilde{\delta}^{-\frac{1}{2}}\abs{\xi}, 
\end{equation}
and for all $x,y\in\mathbb{R}^d$,
\begin{equation}\label{holder_continous}
  \norm{\sigma^R(x)-\sigma^R(y)}\le \tilde{\delta}^{-\frac{1}{2}}\abs{x-y}^\varpi,  
\end{equation}
where $\tilde{\delta}$ is a constant only depend on $\delta$ and $d$. }

\item[{ $(\bf{H}_2^{\sigma^R})$}] There exist two positive constants $\beta$ and $\tilde{\beta}$ such that for all $R\ge 1$,
{\begin{equation}
    \left(\int_{\mathbb{R}^d} \norm{\nabla \sigma^R}^{p_1}\,dx\right)^\frac{1}{p_1}\le \Big(C(d,\delta,p_1)\ +(4\beta I_\sigma(3R)+4\tilde{\beta})\Big), 
\end{equation}}
where $p_1>d$ is a constant and $C(d,\delta,p_1)$ is a constant only depend on $d$, $\delta$ and $p_1$.
\end{itemize}
Let $\overbar{W}$ be an independent copy of the $d$-dimensional standard Wiener process $W$ and let
$$\widetilde{W}:=\begin{bmatrix}
    W\\ 
    \overbar{W}\\
\end{bmatrix}.$$
We can verify that $\widetilde{W}$ is a $2d$-dimensional standard  Wiener process. In SDEs \eqref{eq:newsde}, the coefficients $b^R$ and $\sigma^R$ are defined by
\[
b^R(x):=b(x)\mathds{1}_{\abs{x}\le R},\ \ \sigma^R(x):=[\rho_R \sigma, h_R\bar{\sigma}](x),
\]
where $\bar{\sigma}$ is a matrix defined by 
$$ \bar{\sigma}(x)\equiv \begin{pmatrix}
  \delta^{-\frac{1}{2}} & &\\
  &\ddots&\\
  & & \delta^{-\frac{1}{2}}
\end{pmatrix}_{d\times d}.$$
The cut-off function $h_R$ be defined by 
$$ h_R(x)=\begin{cases}
    0,\quad \abs{x}\le R,\\
    \frac{2}{R^2}({\abs{x}-R})^2,\quad R\le\abs{x}\le \frac{3R}{2},\\
   1-\frac{2}{R^2}(\abs{x}-2R)^2, \quad \frac{3R}{2}<\abs{x}\le 2R,\\
   1,\quad \abs{x}>2R.
  \end{cases} $$ 
It is easy to verify $h_R$ satisfy
$$ h_R(x)=\begin{cases}
  0,\quad \abs{x}\le R,\\
  \in (0,1)  \quad R<\abs{x}\le 2R,\\
  1 \quad \abs{x}>2R,
\end{cases} \quad  
 \abs{\nabla h_R}(x)=\begin{cases}
    0,\quad \abs{x}\le R,\\
    \le \frac{2}{R}  \quad R<\abs{x}\le 2R,\\
    0 \quad \abs{x}>2R.
\end{cases} $$ 
Similarly, we can construct a cut-off function $\rho_R$ satisfy
$$ \rho_R(x)=\begin{cases}
  1,\quad \abs{x}\le 2R,\\
  \in (0,1)  \quad 2R<\abs{x}\le 3R,\\
  0 \quad \abs{x}>3R,
\end{cases} \quad
 \abs{\nabla\rho_R}(x)=\begin{cases}
    0,\quad \abs{x}\le 2R,\\
   \le \frac{2}{R}  \quad 2R<\abs{x}\le 3R,\\
    0 \quad \abs{x}>3R.
\end{cases} $$
Clearly, $\mathbf{(H^{b^R})}$ hold by the definition of $b^R$. Notice that 
\[
\langle \sigma^R(\sigma^R)^\top\xi,\xi\rangle=\rho^2_R\langle \sigma\sigma^\top\xi,\xi\rangle+h^2_R\langle \bar{\sigma}\bar{\sigma}^\top\xi,\xi\rangle,
\]
by the definitions of $\rho_R$, $h_R$, $\bar{\sigma}$ and  assumption $(\bf{H}^\sigma_1)$, we have
\begin{equation}\label{sigmaR-nondegenerate}
   \frac{1}{2}\delta\abs{\xi}^2\le \langle \sigma^R(\sigma^R)^\top\xi,\xi\rangle\le 2\delta^{-1}\abs{\xi}^2, \ \ \forall\,\xi\in\mathbb{R}^d. 
\end{equation}
On the other hand, it is easy to see for all $x,y\in B(2R)\backslash B(R)$,
\begin{equation*}
    \begin{aligned}
         &\abs{h_R(x)-h_R(y)}\le \frac{2}{R}\abs{x-y}\le \frac{2}{R}(4R)^{1-\varpi}\abs{x-y}^\varpi\le 8\abs{x-y}^\varpi,\ \forall\, R\ge 1,\\
    \end{aligned}
\end{equation*}
and for all $x,y\notin B(2R)\backslash B(R)$, we have $\abs{h_R(x)-h_R(y)}\le \abs{x-y}^\varpi,\ \forall\,R\ge 1$. Hence, for all $x,y\in \mathbb{R}^d$, we obtain
\begin{equation}\label{hR-holder}
    \abs{h_R(x)-h_R(y)}\le 8\abs{x-y}^\varpi,\ \forall\,R\ge 1.
\end{equation}
Similarly, we can obtain 
\begin{equation}\label{rhoR-holder}
  \abs{\rho_R(x)-\rho_R(y)}\le 12\abs{x-y}^\varpi,\ \forall\,R\ge 1.  
\end{equation}
Therefore, we have
{\begin{equation}\label{sigmaR-holder}
  \begin{aligned}
     & \norm{\sigma^R(x)-\sigma^R(y)}\\
     \le & \abs{\rho_R(x)-\rho_R(y)}\norm{\sigma(x)}+\abs{\rho_R(y)}\norm{\sigma(x)-\sigma(y)}+\norm{\bar{\sigma}}\abs{h_R(x)-h_R(y)}\\
     \le & \left( 12d\cdot \delta^{-\frac{1}{2}}d^{\frac{1}{2}}+\delta^{-\frac{1}{2}}+12\delta^{-\frac{1}{2}}d^\frac{1}{2} \right)\abs{x-y}^\varpi,
  \end{aligned}
\end{equation}}
where  the last inequality is due to \eqref{hR-holder} and \eqref{rhoR-holder}. Combining \eqref{sigmaR-nondegenerate} with \eqref{sigmaR-holder}, we verified the  $(\bf{H}^{\sigma^R}_1)$. \\
By the definition $\sigma^R=[\rho_R \sigma,h_R\bar{\sigma}]$ and direct computation, we obtain
\begin{equation*}
  \begin{aligned}
     &\int_{\mathbb{R}^d} \norm{\nabla \sigma^R}^{p_1}\,dx= \int_{\mathbb{R}^d} \norm{\nabla [\rho_R\,\sigma,h_R\,\bar{\sigma}]}^{p_1}\,dx\\
     =&\int_{\mathbb{R}^d} \norm{[\nabla \rho_R(x)\,\sigma(x)+\rho_R(x)\,\nabla \sigma(x), \nabla h_R(x)\,\bar\sigma(x)+h_R(x)\,\nabla \bar\sigma(x)]}^{p_1}\,dx\\
     \le &\,4^{p_1} \Bigg\{\int_{B(3R)\backslash B(2R)} \norm{\nabla \rho_R(x)\sigma(x)}^{p_1}\,dx + \int_{B(2R)\backslash B(R)} \norm{\nabla h_R(x)\bar\sigma(x)}^{p_1}\,dx\\
     &\quad\quad\ + \int_{B(3R)}\norm{\nabla \sigma}^{p_1}\,dx \Bigg\}\\
     :=\, & 4^{p_1}\,(J_1+J_2+J_3).
  \end{aligned}
\end{equation*}
Note that $\abs{\nabla \rho_R}\le \frac{2}{R}$ in $B(3R)\backslash B(2R)$, $\abs{\nabla h_R}\le \frac{2}{R}$ in $B(2R)\backslash B(R)$  and $(\bf{H}_2^\sigma)$, there exist a constant $C(d,\delta,p_1)$ only depend on $d$, $\delta$ and $p_1$ such that for all $R\ge 1$, 
\begin{equation*}
    \begin{aligned}
         &  J_1\le  \int_{B(3R)\backslash B(2R)} C(d)\left(\frac{1}{R}\delta^{-\frac{1}{2}}d^{\frac{1}{2}}\right)^{p_1}\,dx\le  C(d,\delta,p_1)R^{d-p_1}\le C(d,\delta,p_1),\\ 
         &   J_2 \le \int_{B(2R)\backslash B(R)} C(d) (\frac{1}{R}\delta^{-\frac{1}{2}})^{p_1}\,dx\le  C(d,\delta,p_1)R^{d-p_1}\le C(d,\delta,p_1),\\ 
         & J_3\le \int_{B(3R)}\norm{\nabla \sigma(x)}^{p_1}\,dx\le (\beta I_\sigma(3R)+\tilde{\beta})^{p_1}.
  \end{aligned}
\end{equation*}
Together, $J_1$, $J_2$ and $J_3$ imply $(\bf{H}_2^{\sigma^R})$.

\section{Krylov's estimate and Khasminskii's estimate}\label{sec4}
In this section, we shall prove Krylov's estimate and Khasminskii's estimate. We need the following result about elliptic PDEs \eqref{studyed-pde}. 

\begin{theorem}\label{studied-pde-theorem}
    Suppose $\sigma^R$ satisfies { $(\bf{H}_1^{\sigma^R})$}, $p\in (1,\infty)$, then for any $f\in L^p(\mathbb{R}^d)$, there exists a unique $u\in W^{2,p}(\mathbb{R}^d)$ such that
    \begin{equation}\label{studyed-pde}
        L^{\sigma^R(x)}u - \lambda u = f,
    \end{equation}
    where 
    $$
    L^{\sigma^R(x)} u(x):= \frac{1}{2}\sum^{d}_{i,j,k=1}(\sigma^R)_{ik}(x)(\sigma^R)_{jk}(x)\partial_i \partial_j u(x)
    $$ 
    and $\lambda >C$ \ $(C=C(d,\varpi,\tilde{\delta},p)\ge 2$ is a constant $)$.
Furthermore, for a $C_1=C_1(d,\varpi ,\tilde{\delta},p)>0$, 
    \begin{equation}\label{krylov-lemma}
        \norm{u}_{2,p}\le C_1\norm{f}_p.
    \end{equation}
    Moreover, for  any $\alpha\in [0,2)$ and $p'\in [1,\infty]$ with $\frac{d}{p}<2-\alpha+\frac{d}{p'}$, 
\begin{equation}\label{para-space}
    \norm{u}_{\alpha,p'}\le C_2\, \lambda^{(\alpha-2+\frac{d}{p}-\frac{d}{p'})/2}\norm{f}_p,
\end{equation}
    where $C_1(d,\varpi,\tilde{\delta},p)$ and $C_2(d,\varpi ,\tilde{\delta},p,\alpha,p')>0$  are  both independent of $\lambda$.
\end{theorem}
We believe that Theorem \ref{studied-pde-theorem} is standard although we do not find them in any reference. In \cite{zhang2018}, authors proved Theorem \ref{studied-pde-theorem} hold true when $\sigma^R\equiv I$. 
For convenience of the reader, we combine  \cite{zhang2018} with \cite{zhang2016} to give a detailed proof in Appendix.

{In order to prove Krylov's estimate and Khasminskii's estimate, we need to solve the following elliptic equation:
\begin{equation}\label{eq:R-PDE}
    \begin{aligned}
         & (L^{\sigma^R(x)}-\lambda) u^R +b^R\cdot\nabla u^R = f, \quad \lambda\ge  \lambda^{b^R},
    \end{aligned}
\end{equation} 
where $f\in L^p(\mathbb{R}^d)$ and $\lambda^{b^R}>1$ is a constant depend on $C_2,d,p_1$ and $\norm{b^R}_{p_1}$.

\begin{lemma}
    If $\norm{b^R}_{p_1}<\infty$ and $(\bf{H}_1^{\sigma^R})$ hold, then for any $p\in (\frac{d}{2}\vee 1,p_1]$, we can find a constant 
\begin{equation}
    \begin{aligned}
         & \lambda^{b^R}=\left(2C_2\norm{b^R}_{p_1}\right)^{2(1-\frac{d}{p_1})^{-1}}
    \end{aligned}
\end{equation}
    such that for any $f\in L^p(\mathbb{R}^d)$, there exists a unique solution $u^R\in W^{2,p}(\mathbb{R}^d)$ to equation \eqref{eq:R-PDE} and
    \begin{equation}
        \begin{aligned}
             \norm{u^R}_{2,p}\le 2C_1\norm{f}_p, \ \ 
             \lambda^{(2-\alpha+\frac{d}{p'}-\frac{d}{p})/2}\norm{u^R}_{\alpha,p'}\le 2C_2 \norm{f}_p\ (\lambda\ge \lambda^{b^R}),
        \end{aligned}
    \end{equation}
    where $C_1$ and $C_2$ are two constants in Theorem \ref{studied-pde-theorem},  $\alpha\in [0,2)$ and $p'\in[1,\infty]$ with $(2-\alpha + \frac{d}{p'}-\frac{d}{p})>0$.
\end{lemma}
\begin{proof}
    By Theorem \ref{studied-pde-theorem}, for any $\tilde{f}\in L^p(\mathbb{R}^d)$, we have
    \begin{equation}\label{eq:no-b^R}
        \begin{aligned}
             \norm{(\lambda-L^{\sigma^R(x)})^{-1} \tilde{f}}_{2,p}\le C_1 \norm{\tilde{f}}_p,\quad \lambda^{(2-\alpha + \frac{d}{p'}-\frac{d}{p})/2}\norm{(\lambda-L^{\sigma^R(x)})^{-1} \tilde{f}}_{\alpha,p'}\le C_2 \norm{\tilde{f}}_p,
        \end{aligned}
    \end{equation}
where $\lambda >C\ \ (C>2)$, $(2-\alpha + \frac{d}{p'}-\frac{d}{p})>0$ and $C_1, C_2$ do not depend on $\lambda$. \\
Since 
$\lambda^{b^R}=(2C_2\norm{b^R}_{p_1})^{2p_1/(p_1-d)}$, it is easy to see for any $\lambda \ge \lambda^{b^R}$,
$$C_2\lambda^{(\frac{d}{p_1}-1)/2}\norm{b^R}_{p_1}\le \frac{1}{2}.$$ 
Let $u_0=0$ and for $n\in \mathbb{N}$ define
\[
u^R_n:= ( L^{\sigma^R(x)} -\lambda)^{-1}(f-b^R\cdot \nabla u^R_{n-1}).    
\]
By  \eqref{eq:no-b^R} and replace  $(\Delta-\lambda)^{-1}$ with $( L^{\sigma^R(x)} -\lambda)^{-1}$ in the proof of \cite[Theorem $3.3$ ${(ii)}$]{zhang2018},
we completed the proof.
\end{proof}

Now, we provide the main result of this section. 
\begin{theorem}
    If $\norm{b^R}_{p_1}<\infty$ and $\mathbf{(H^{\sigma^R}_1)}$ hold and $\{X_s^R\}_{s\in[0,T]}$ is a solution of SDE \eqref{eq:newsde}, then for any $0\le t_0<t_1\le T$, $f\in L^p(\mathbb{R}^d)$ \ $(p>\frac{d}{2}\vee 1)$, we have
\begin{equation}\label{lambda-b-R}
    \mathbb{E}^{\mathscr{F}_{t_0}}\left[ \int^{t_1}_{t_0} f(X^R_s(x))\,ds \right]\le 4C_2\left([T\lambda^{b^R}]^\frac{d}{2p}+[T\lambda^{b^R}]^{\frac{d}{2p}-1}\right)(t_1-t_0)^{1-\frac{d}{2p}}\norm{f}_p,
\end{equation}
where $C_2$ is the constant in Theorem \ref{studied-pde-theorem},  $\lambda^{b^R}=(2C_2\norm{b^R}_{p_1})^{2p_1/(p_1-d)}$.
Moreover, for any $a>0$  we have
    \begin{equation}
        \begin{aligned}
             \mathbb{E}\left[ \exp\left( a\int^T_0 \abs{f(X^R_s(x))}\,ds \right) \right]\le           
             e\cdot\exp\left(T\left[ \frac{4a C_2 \left([T\lambda^{b^R}]^{\frac{d}{2p}}+[T\lambda^{b^R}]^{\frac{d}{2p}-1}\right)\norm{f}_p}{1-e^{-1}} \right]^{(1-\frac{d}{2p})^{-1}}\right).
        \end{aligned}
    \end{equation}
\end{theorem}

\begin{proof}
The proof be divided into three steps.

    {\bf Step (i)}  We replace  $(\Delta-\lambda)^{-1}$ with $( L^{\sigma^R(x)} -\lambda)^{-1}$ in the proof of Theorem $3.4$ of Zhang and Zhao \cite{zhang2018}. Notice that 
 $$  \lambda^{b^R}=\left(2C_2\norm{b^R}_{p_1}\right)^{2(1-\frac{d}{p_1})^{-1}}$$
is enough to ensure $C_2\lambda^{(d/p_1-1)/2}\norm{b^R}_{p_1}\le \frac{1}{2}$ for all $\lambda \ge \lambda^{b^R}$.  Repeating the proof of Theorem $3.4$ (ii) of Zhang and Zhao [25],   for all $\tilde{\lambda}\ge \lambda^{b^R}$, we obtain
    \begin{equation}\label{Krylov-estimate}
        \begin{aligned}
              \mathbb{E}^{\mathscr{F}_{t_0}}\left[ \int^{t_1}_{t_0}f(X^R_s(x))\,ds \right]&\le \tilde{\lambda} (t_1-t_0)\norm{u^R}_\infty+ 2\norm{u^R}_\infty\\
              &\le 2C_2(t_1-t_0)\tilde{\lambda}^{\frac{d}{2p}}\norm{f}_p + 4C_2\tilde{\lambda}^{(\frac{d}{2p}-1)}\norm{f}_p.\\
        \end{aligned}
    \end{equation}
Let  $\kappa=T\lambda^{b^R}$ and $\tilde{\lambda}=\kappa(t_1-t_0)^{-1}$. Due to $0\le t_0<t_1\le T$, we have $\tilde{\lambda}\ge \lambda^{b^R}$. Taking $\tilde{\lambda}=\kappa(t_1-t_0)^{-1}$ into \eqref{Krylov-estimate} we proved the Krylov's estimate \eqref{lambda-b-R}.         

{\bf Step (ii)} Taking $0\le t_0< t_1<\infty$ satisfy
\begin{equation}\label{t1-t0}
    t_1-t_0=\left( \frac{1-e^{-1}}{4aC_2(\kappa^{\frac{d}{2p}}+\kappa^{\frac{d}{2p}-1})\norm{f}_p} \right)^{(1-\frac{d}{2p})^{-1}}.
\end{equation}
If $t_1-t_0\le T$ in \eqref{t1-t0}, by the Corollary 3.5 in Zhang and Zhao \cite{zhang2018}, we have
\begin{equation}
        \begin{aligned}
             \mathbb{E}^{\mathscr{F}_{t_0}}\left[ \left(\int^{t_1}_{t_0}\abs{f(X^R_s(x))}\,ds\right)^n \right]\le n!\left(\frac{1-e^{-1}}{a}\right)^n.
        \end{aligned}
    \end{equation}
Since $e^x=\sum^\infty_{n=0} \frac{1}{n!} x^n$, we have
    \begin{equation}\label{eq:exp-short-time}
        \begin{aligned}
             & \mathbb{E}^{\mathscr{F}_{t_0}}\left[ \exp\left\{ a\int^{t_1}_{t_0}\abs{f(X^R_s(x))}\,ds \right\}\right]\\
             =&\mathbb{E}^{\mathscr{F}_{t_0}}\left[\sum^\infty_{n=0}\frac{1}{n!}\left(a\int^{t_1}_{t_0}\abs{f(X^R_s(x))}\,ds \right)^n  \right]\\
             =&\sum^\infty_{n=0}\frac{1}{n!}\mathbb{E}^{\mathscr{F}_{t_0}}\left[\left(a\int^{t_1}_{t_0}\abs{f(X^R_s(x))}\,ds \right)^n \right]\\
             \le &\sum^\infty_{n=0} (1-e^{-1})^n=e.
        \end{aligned}
    \end{equation}

   {\bf Step (iii)} Finally, by virtual of the estimate \eqref{eq:exp-short-time}, we obtain
   \begin{equation}
    \begin{aligned}
         & \mathbb{E}\left[ \exp\left\{ a\int^{T}_{0}\abs{f(X^R_s(x))}\,ds \right\}\right]\\
         \le&\mathbb{E}\left[ \exp\left\{ a\sum^{\left\lfloor M\right\rfloor+1 }_{i=1}\int^{t_i}_{t_{i-1}}\abs{f(X^R_s(x))}\,ds  \right\}\right]\\
         =&\mathbb{E}\left[\prod^{\left\lfloor M\right\rfloor+1 }_{i=1} \exp\left\{ a\int^{t_i}_{t_{i-1}}\abs{f(X^R_s(x))}\,ds  \right\}\right]\\
         =&\mathbb{E}\left[\prod^{\left\lfloor M\right\rfloor}_{i=1} \exp\left\{ a\int^{t_i}_{t_{i-1}}\abs{f(X^R_s(x))}\,ds  \right\}\mathbb{E}^{\mathscr{F}_{t_{{\left\lfloor M\right\rfloor}}}}\left[ \exp\left\{ a\int^{t_{\left\lfloor M\right\rfloor+1}}_{t_{{\left\lfloor M\right\rfloor}}}\abs{f(X^R_s(x))}\,ds  \right\} \right]\right]\\
         \le &e\cdot\mathbb{E}\left[\prod^{\left\lfloor M\right\rfloor}_{i=1} \exp\left\{ a\int^{t_i}_{t_{i-1}}\abs{f(X^R_s(x))}\,ds  \right\}\right]\le e^{M+1},
    \end{aligned}
\end{equation} 
where $M=\frac{T}{t_1-t_0}$ and $0\le t_0<t_1<\cdots<t_{\left\lfloor M\right\rfloor+1}=T$ satisfies $t_0-0\le t_1-t_0$,  $t_i-t_{i-1}=t_1-t_0\ (i=1,\cdots,\left\lfloor M\right\rfloor+1)$.  \\
If $t_1-t_0>T$ in \eqref{t1-t0}, it is obvious that
$$
\mathbb{E}\left[\int^T_0 f(X^R_s(x)\,ds)\right]\le \frac{1-e^{-1}}{a},
$$
by a similar argument, we have 
\[
\mathbb{E}\left[ \exp\left\{ a\int^{T}_{0}\abs{f(X^R_s(x))}\,ds \right\}\right]\le e.
\]
We completed the proof.
\end{proof}
}

In particular, in the proofs of Lemma \ref{lemma-u^R-estimate} and Theorem \ref{th:exp-bounded}, replacing  $\lambda^{b^R}$ with  $\lambda^R=\big(4C^2_2(\beta I_b(R)+\tilde{\beta})^2\big)^{p_1/(p_1-d)}$, we can obtain the following lemma and theorem:

\begin{lemma}\label{lemma-u^R-estimate}
    If $\mathbf{(H^{b^R})}$ and $(\bf{H}_1^{\sigma^R})$ hold, then for any $p\in (\frac{d}{2}\vee 1,p_1]$, we can find a constant 
\begin{equation}\label{definition_of_lambda^R}
    \begin{aligned}
         & \lambda^R=\big(4C^2_2(\beta I_b(R)+\tilde{\beta})^2\big)^{(1-\frac{d}{p_1})^{-1}}
    \end{aligned}
\end{equation}
    such that for any $f\in L^p(\mathbb{R}^d)$, there exists a unique solution $u^R\in W^{2,p}(\mathbb{R}^d)$ to equation \eqref{eq:R-PDE} and
    \begin{equation}\label{lamada-estimate}
        \begin{aligned}
             \norm{u^R}_{2,p}\le 2C_1\norm{f}_p, \ \ 
             \lambda^{(2-\alpha+\frac{d}{p'}-\frac{d}{p})/2}\norm{u^R}_{\alpha,p'}\le 2C_2 \norm{f}_p\ (\lambda\ge \lambda^R),
        \end{aligned}
    \end{equation}
    where $C_1$ and $C_2$ are two constants in Theorem \ref{studied-pde-theorem},  $\alpha\in [0,2)$ and $p'\in[1,\infty]$ with $(2-\alpha + \frac{d}{p'}-\frac{d}{p})>0$.
\end{lemma}

\begin{theorem}\label{th:exp-bounded}
    If $(\mathbf{H^{b^R}})$ and $\mathbf{(H^{\sigma^R}_1)}$ hold and $\{X_s^R\}_{s\in[0,T]}$ is a solution of SDE \eqref{eq:newsde}, then for any $0\le t_0<t_1\le T$, $f\in L^p(\mathbb{R}^d)$ \ $(p>\frac{d}{2}\vee 1)$, we have
\begin{equation}\label{krylov-t1-t0}
    \mathbb{E}^{\mathscr{F}_{t_0}}\left[ \int^{t_1}_{t_0} f(X^R_s(x))\,ds \right]\le 4C_2([T\lambda^R]^\frac{d}{2p}+[T\lambda^R]^{\frac{d}{2p}-1})(t_1-t_0)^{1-\frac{d}{2p}}\norm{f}_p,
\end{equation}
where $C_2$ is the constant in Theorem \ref{studied-pde-theorem},  $\lambda^R=\big(4C^2_2(\beta I_b(R)+\tilde{\beta})^2\big)^{p_1/(p_1-d)}$.
Moreover, for any $a>0$  we have
    \begin{equation}\label{exp-estimate}
        \begin{aligned}
             &\mathbb{E}\left[ \exp\left( a\int^T_0 \abs{f(X^R_s(x))}\,ds \right) \right]\\
\le & e\cdot\exp\left(T\left[ \frac{4a C_2 ([T\lambda^R]^{\frac{d}{2p}}+[T\lambda^R]^{\frac{d}{2p}-1})\norm{f}_p}{1-e^{-1}} \right]^{(1-\frac{d}{2p})^{-1}}\right).
        \end{aligned}
    \end{equation}
\end{theorem}

\begin{corollary}[Generalized It\^o's formula]
    If $(\mathbf{H^{b^R}})$ and $\mathbf{(H^{\sigma^R}_1)}$ hold and $\{X_s^R\}_{s\in[0,T]}$ is a solution of SDE \eqref{eq:newsde}, then for any $f\in W^{2,p}(\mathbb{R}^d)$ with $p>\frac{d}{2}\vee 1$, we have
    \begin{equation}\label{Ito-formula-Sobolev}
        \begin{aligned}
             & f(X^R_t)=f(x)+\int^t_0 (L^{\sigma^R(x)} f + b^R\cdot \nabla f)(X^R_s)\,ds + \int^t_0 \langle \nabla f(X^R_s), \sigma^R(X^R_s)\,d\widetilde{W}_s\rangle.
        \end{aligned}
    \end{equation}
\end{corollary}
\begin{proof}
We just need to consider the case $p\in (d,p_1]$ since $W^{2,p}\hookrightarrow W^{2,p_1}$ when $p>p_1$.\\
By H\"older's inequality  and Sobolev's embedding theorem, we have
    \begin{equation}\label{Ito-drift}
        \begin{aligned}
             & \norm{L^{\sigma^R(x)} f + b^R\cdot \nabla f}_p \lesssim \norm{f}_{2,p}+\norm{b^R}_{p_1}\norm{\nabla f}_{\frac{p_1 p}{p_1-p}}\lesssim \norm{f}_{2,p}.
        \end{aligned}
    \end{equation}
Let $\varphi$ be a nonnegative smooth function with compact support in the unit ball of $\mathbb{R}^d$ and $\int_{\mathbb{R}^d} \varphi(x)\,dx=1$. Set $\varphi_n(x):=n^d\varphi(nx)$,  $f_n:=f*\varphi_n$ and applying It\^o formula to $f_n$. By \eqref{Ito-drift}, we have
\begin{equation}\label{Ito-convergence-1}
    \begin{aligned}
        & \norm{L^{\sigma^R(x)} (f-f_n) + b^R\cdot \nabla (f-f_n)}_p \lesssim \norm{f-f_n}_{2,p}\rightarrow 0. 
    \end{aligned}
\end{equation}
Let $\bar{p}=\frac{dp}{2(d-p)}$, we have
\begin{equation}\label{Ito-convergence-2}
    \begin{aligned}
         & \mathbb{E}\left| \int^t_0 \langle (\nabla f(X^R_s)- \nabla f_n(X^R_s)), \sigma^R(X^R_s)\,d\widetilde{W}_s\rangle\right|^2 \\
         \lesssim& \norm{\sigma^R}^2_\infty\mathbb{E}\int^t_0 \abs{\nabla f(X^R_s)- \nabla f_n(X^R_s)}^2\,ds\\
          \lesssim&\norm{\abs{\nabla f-\nabla f_n}^2}_{\bar{p}}\lesssim\norm{f-f_n}^2_{1,2\bar{p}}\\
          \lesssim&\norm{f-f_n}^2_{2,p}\rightarrow 0,
    \end{aligned}
\end{equation}
where the second inequality is due to Krylov's estimate \eqref{krylov-t1-t0} and the last inequality is due to Sobolev's embedding theorem.
Together, \eqref{Ito-convergence-1} and \eqref{Ito-convergence-2} imply \eqref{Ito-formula-Sobolev}.
\end{proof}

\section{Zvonkin's transformation}\label{sec5}
Let $u^R$ solve the following PDE
\begin{equation}
    \begin{aligned}
         (L^{\sigma^R(x)} -\lambda)u^R + b^R\cdot \nabla u^R = -b^R.
    \end{aligned}
\end{equation}
By Lemma \ref{lemma-u^R-estimate}, we have
\begin{equation}\label{uR2p1}
    \begin{aligned}
         & \norm{u^R}_{2,p_1}\le 2C_1\norm{b^R}_{p_1},\quad \lambda^{(1-\frac{d}{p_1})/2}\norm{u^R}_{1,\infty}\le 2C_2\norm{b^R}_{p_1}\ (\lambda\ge \lambda^R).
    \end{aligned}
\end{equation}
Let $\lambda^R_H=\gamma\lambda^R$ and $ \gamma^{(\frac{d}{2p_1}-\frac{1}{2})}=\frac{1}{2}$, it is easy to check
\begin{equation}\label{u^R-bounded}
    \begin{aligned}
         \norm{\nabla u^R}_\infty\le \norm{u^R}_{1,\infty}\le \gamma^{(\frac{d}{2p_1}-\frac{1}{2})}=\frac{1}{2}.
    \end{aligned}
\end{equation}
Define 
\begin{equation}
    \begin{aligned}
         \Phi_R(x):=x+u^R(x),
    \end{aligned}
\end{equation}
then 

\begin{equation}
    \begin{aligned}
         L^{\sigma^R(x)} \Phi_R + b^R\cdot\nabla \Phi_R = \lambda u^R.
    \end{aligned}
\end{equation}
By \eqref{u^R-bounded}, for all $\lambda\ge \lambda^R_H$, we have
\begin{equation}\label{uR-1-infnity}
    \norm{u^R}_\infty\le \frac{1}{2},\ \ \norm{\nabla u^R}_\infty\le \frac{1}{2}.
\end{equation}
By the definition of $\Phi_R(x)$ and \eqref{uR-1-infnity}, we have
\[
\lim_{\abs{x}\rightarrow \infty}\abs{\Phi_R(x)}=\infty,\ \ \frac{1}{2}\abs{x-y}\le \abs{\Phi_R(x)-\Phi_R(y)}\le 2\abs{x-y}.
\]
Therefore, by Theorem \ref{Hadamard's Theorem}, we obtain $\Phi_R:\mathbb{R}^d\rightarrow \mathbb{R}^d$ is a $C^1$-diffeomorphism and 
\begin{equation}\label{grad_bound}
    \begin{aligned}
         \norm{\nabla\Phi_R}_\infty\le 2,\quad \norm{\nabla{\Phi_R^{-1}}}_\infty\le 2.
    \end{aligned}
\end{equation}

\begin{theorem}\label{th:zvonkin-SDE}
Let $Y^R_t:=\Phi_R(X_t^R)$, then  $X^R_t$ solve equation \eqref{eq:newsde} if and only if $Y^R_t$ solves
\begin{equation}\label{eq:Zvonkin-SDE}
    \left\{\begin{aligned}
        & dY^R_t = \tilde{b}^R(Y^R_t)\,dt+ \tilde{\sigma}^R(Y^R_t)\,d\widetilde{W}_t,\ \ t\in [0,T],\\
        & Y^R_0=\Phi_R(x),
\end{aligned}\right.
\end{equation}
where $\tilde{b}^R(y):=\lambda u^R\circ \Phi_R^{-1}(y)$ and $\tilde{\sigma}^R(y):={(\nabla \Phi_R(\cdot)\sigma^R(\cdot))\circ \Phi_R^{-1}(y)}$.
\end{theorem}
 \begin{proof}
      Applying  It\^o formula \eqref{Ito-formula-Sobolev} to $\Phi_R(X^R_t)$, we obtain
\[
\Phi_R(X^R_t)=\Phi_R(x)+\lambda\int^t_0 u^R(X^R_s)\,ds+\int^t_0 \nabla\Phi_R(X^R_s)\sigma^R(X^R_s)\,d\widetilde{W}_s.
\]
Noticing that $Y^R_t=\Phi_R(X_t^R)$, we obtain $Y^R_t$ solves \eqref{eq:Zvonkin-SDE}. Similarly, applying It\^o formula \eqref{Ito-formula-Sobolev} to $\Phi^{-1}_R(Y^R_t)$, we completed the proof.
 \end{proof}

\section{The proof of Theorem \ref{main-theorem}}\label{sec6}
\begin{proof}
In this section the letter $\mathbf{C}$ and $\mathbf{\widetilde{C}}$ will denote some unimportant constant whose value  is independent of $R$ and may change in different places. Whose dependence on parameters can be traced from the context. We also use $\mathbf{C}(T)$ and  $\mathbf{C}(N)$ to emphasize the constant  $\mathbf{C}$  depend on $T$ and $N$ respectively. 

{Firstly, we prove SDE \eqref{eq:newsde}} exists a unique strong solution.

\begin{theorem}\label{th:XR-exist}
   Under $(\mathbf{H^{b^R}_1})$, $(\mathbf{H^{\sigma^R}_1})$ and  $(\mathbf{H^{\sigma^R}_2})$, for all $x\in \mathbb{R}^d$, the SDE\ \eqref{eq:newsde}  exists a unique strong solution.
\end{theorem}
\begin{proof}
    By Theorem \ref{th:zvonkin-SDE}, we only need to prove SDE \eqref{eq:Zvonkin-SDE} exists a unique strong solution.
By the definition of $\tilde{b}^R$, $\tilde{\sigma}^R$ and Lemma \ref{lemma-u^R-estimate}, for all $\lambda\ge \lambda^R_H$, we have
\begin{equation}\label{observe}
    \begin{aligned}
        &\norm{\tilde{b}^R}_\infty\le \frac{1}{2}\lambda,\ \ \norm{\nabla \tilde{b}^R}_\infty\le \lambda,\quad \norm{\tilde{\sigma}^R}_\infty\le 2\norm{\sigma^R}_\infty,\\
    \end{aligned}
\end{equation}
Note that $\tilde{b}^R$ and $\tilde{\sigma}^R$ are both continuous and bounded. By Yamada-Watanabe's theorem, we only need to show the pathwise uniqueness. Performing the same procedure in \cite[Theorem $3.1$]{zhang2018}, we completed the proof. 
\end{proof}

{\begin{lemma}\label{lm:XR-two-point}
   Under $(\mathbf{H^{b^R}})$, $(\mathbf{H^{\sigma^R}_1})$ and $(\mathbf{H^{\sigma^R}_2})$, let $\{X^R_s(x)\}_{s\in[0,T]}$ and $\{X^R_s(y)\}_{s\in[0,T]}$ are two solutions of SDE \eqref{eq:newsde} with initial conditions $X^R_0(x)=x$ and $X^R_0(y)=y$ respectively, then for any $\alpha\in \mathbb{R}$, we have
\begin{equation}\label{eq:XR-x-y}
    \begin{aligned}
         \mathbb{E}\left[\abs{X_t^R(x)-X_t^R(y)}^\alpha\right]\le \widetilde{\mathbf{C}}\left(\exp\left(\widetilde{\mathbf{C}}\,(\lambda^R)^{\frac{p_1}{p_1-d}}\right)\right)\abs{x-y}^\alpha,
    \end{aligned}
\end{equation}
{
\begin{equation}\label{1-plus-X}
    \begin{aligned}
         & \mathbb{E}\left[\left(1+\abs{X^R_t(x)}^2\right)^\alpha\right]\le \widetilde{\mathbf{C}}\left( \exp\big(\widetilde{\mathbf{C}}\,\lambda^R\big) \right)\left(1+\abs{x}^2\right)^\alpha,
    \end{aligned}
\end{equation}
}
and for all $p\ge 2$,
\begin{equation}\label{X^R_one_poit_estimate}
    \begin{aligned}
         &  \mathbb{E}\left[ \sup_{0\le s\le t}\abs{X^R_s(x)}^p\right]\le \widetilde{\mathbf{C}}\,(1+\abs{x}^p+(\lambda^R)^p), 
    \end{aligned}
\end{equation}
\begin{equation}\label{sup_xy}
    \begin{aligned}
       &\mathbb{E}\left[\sup_{0\le s\le t}\abs{X_t^R(x)-X_t^R(y)}^p\right]\le \widetilde{\mathbf{C}}\left(\exp\left(\widetilde{\mathbf{C}}\,(\lambda^R)^{\frac{p_1}{p_1-d}}\right)\right)\abs{x-y}^p,
    \end{aligned}
\end{equation}
where $\widetilde{\mathbf{C}}$ is independent of $\beta$, $\tilde{\beta}$ and $R$.
\end{lemma}
\begin{proof}
    For  $\Phi_R(x)\neq \Phi_R(y)$, take $ 0<\varepsilon<\abs{\Phi_R(x)-\Phi_R(y)}$ and set
    \begin{equation}
        \begin{aligned}
             & \tau_\varepsilon := \inf\{ \abs{Y^R_t(\Phi_R(x))-Y^R_t(\Phi_R(y))}\le \epsilon \}.
        \end{aligned}
    \end{equation}
For convenience, we define $Z^R_t := Y^R_t(\Phi_R(x)) -Y^R_t(\Phi_R(y))$ where $\{Y^R_s(\Phi_R(x))\}_{s\in[0,T]}$ and $\{Y^R_s(\Phi_R(y))\}_{s\in[0,T]}$ are the solutions of SDE \eqref{eq:Zvonkin-SDE} with initial conditions $Y^R_0(\Phi_R(x))=\Phi_R(x)$ and $Y^R_0(\Phi_R(y))=\Phi_R(y)$ respectively.\\
By It\^o formula, we have
\begin{equation}\label{eq:ito-formula}
    \begin{aligned}
          \abs{Z^R_{t\wedge \tau_\varepsilon}}^\alpha= &\abs{\Phi_R(x)-\Phi_R(y)}^\alpha + \int^{t\wedge \tau_\varepsilon}_0\alpha\abs{Z^R_s}^{\alpha-2}\langle Z^R_s, (\tilde{\sigma}^R(Y^R_s(x))-\tilde{\sigma}^R(Y^R_s(y)))\,d\widetilde{W}_s \rangle+\\        
          &\int^{t\wedge \tau_\varepsilon}_0\alpha\abs{Z^R_s}^{\alpha-2}\langle Z^R_s, (\tilde{b}^R(Y^R_s(x))-\tilde{b}^R(Y^R_s(y))) \rangle\,ds+\\ 
          &\int^{t\wedge \tau_\varepsilon}_0\frac{\alpha}{2}\abs{Z^R_s}^{\alpha-2}\norm{\tilde{\sigma}^R(Y^R_s(x))-\tilde{\sigma}^R(Y^R_s(y))}^2\,ds+\\
         &\int^{t\wedge \tau_\varepsilon}_0\frac{\alpha(\alpha-2)}{2} \abs{Z^R_s}^{\alpha-4}\abs{(\tilde{\sigma}^R(Y^R_s(x))-\tilde{\sigma}^R(Y^R_s(y)))^\top Z^R_s}^2\,ds.
    \end{aligned}
\end{equation}
Set 
\begin{equation}\label{de:B}
   \mathbf{B}_s :=\frac{\alpha\big(\tilde{\sigma}^R(Y^R_s(x))-\tilde{\sigma}^R(Y^R_s(y))\big)^\top Z^R_s}{\abs{Z_s^R}^2}
\end{equation}
and 
\begin{equation}\label{de:A}
\begin{aligned}
 \mathbf{A}_s :=&\frac{\alpha\langle Z^R_s,(\tilde{b}^R(Y^R_s(x))-\tilde{b}^R(Y^R_s(y))) \rangle}{\abs{Z^R_s}^2}+\frac{\frac{\alpha}{2}\norm{\tilde{\sigma}^R(Y^R_s(x))-\tilde{\sigma}^R(Y^R_s(y))}^2}{\abs{Z^R_s}^2}\\
     &+\frac{\frac{\alpha(\alpha-2)}{2}\abs{\tilde{\sigma}^R(Y^R_s(x))-\tilde{\sigma}^R(Y^R_s(y)))^\top Z^R_s}^2}{\abs{Z_s^R}^4}.  
\end{aligned} 
\end{equation}
By \eqref{eq:ito-formula}, we have
\begin{equation*}
    \begin{aligned}
         & \abs{Z^R_{t\wedge \tau_\varepsilon}}^\alpha=\abs{\Phi_R(x)-\Phi_R(y)}^\alpha+\int_0^{t\wedge \tau_{\varepsilon}} & \abs{Z^R_{s\wedge \tau_\varepsilon}}^\alpha\left( \mathbf{A}_s\,ds +\mathbf{B}_s\,d\widetilde{W}_s \right).
    \end{aligned}
\end{equation*}
By the Dol\'eans-Dade's exponential, we obtain
\begin{equation}\label{D_D_formula}
    \begin{aligned}
         &  \abs{Z^R_{t\wedge \tau_\varepsilon}}^\alpha=\abs{\Phi_R(x)-\Phi_R(y)}^\alpha\exp\left( \int_0^{t\wedge \tau_{\varepsilon}} \mathbf{B}_s\,d\widetilde{W}_s -\frac{1}{2}\int_0^{t\wedge \tau_{\varepsilon}}\abs{\mathbf{B}_s}^2\,ds+\int_0^{t\wedge \tau_{\varepsilon}} \mathbf{A}_s\,ds\right).
    \end{aligned}
\end{equation}
By the definitions of $\tilde{b}^R$ and $\tilde{\sigma}^R$ in Theorem \ref{th:zvonkin-SDE} and Lemma \ref{Max-inequality} (i), it is easy to see
\begin{equation}\label{tilde-sigma-minus}
    \begin{aligned}
         \abs{\tilde{\sigma}^R(x)-\tilde{\sigma}^R(y)}&\le C_d\abs{x-y}\left( \mathcal{M}\abs{\nabla {\sigma^R}}(\Phi_R^{-1}(x))+ \mathcal{M}\abs{\nabla {\sigma^R}}(\Phi_R^{-1}(y))\right)\\
         &\quad+C_d\abs{x-y}\left( \mathcal{M}\abs{\nabla^2 u^R}(\Phi_R^{-1}(x))+ \mathcal{M}\abs{\nabla^2 u^R}(\Phi_R^{-1}(y))\right),
    \end{aligned}
\end{equation}
and
\begin{equation}\label{tilde-b-minus}
    \begin{aligned}
           &\abs{\tilde{b}^R(x)-\tilde{b}^R(y)}=\abs{\lambda u^R\circ \Phi_R^{-1}(x)-\lambda u^R\circ \Phi_R^{-1}(y)}\\ 
\le& \lambda C_d\abs{\Phi_R^{-1}(x)-\Phi_R^{-1}(y)}\left( \mathcal{M}\abs{\nabla u^R}(\Phi_R^{-1}(x))+\mathcal{M}\abs{\nabla u^R}(\Phi_R^{-1}(y)) \right)\\ 
\le& \lambda C_d \abs{x-y}\left( \mathcal{M}\abs{\nabla u^R}(\Phi_R^{-1}(x))+\mathcal{M}\abs{\nabla u^R}(\Phi_R^{-1}(y)) \right).\\ 
\end{aligned}
\end{equation}
Firstly, we  shall prove that for any $\mu>0$,
\begin{equation*}
    \begin{aligned}
         & \mathbb{E}\left[  \exp\left( \mu\int^{T\wedge \tau_\varepsilon}_0\abs{\mathbf{B}_s}^2\,ds \right) \right]\le C(e)\cdot\exp\left( {\widetilde{\mathbf{C}}}\,[\lambda^R]^{(1-\frac{d}{p_1})^{-1}} \right),
    \end{aligned}
\end{equation*}
and
\begin{equation*}
    \begin{aligned}
         & \mathbb{E}\left[  \exp\left( \mu\int^{T\wedge \tau_\varepsilon}_0\abs{\mathbf{A}_s}\,ds \right) \right]\le C(e)\cdot\exp\left( {\widetilde{\mathbf{C}}}\,[\lambda^R]^{(1-\frac{d}{p_1})^{-1}} \right).
    \end{aligned}
\end{equation*}
Combine the definitions of \eqref{de:A}, \eqref{de:B} with \eqref{tilde-sigma-minus}, \eqref{tilde-b-minus}, we only need to estimate
\begin{equation}\label{eq:max1}
    M_1:=\mathbb{E}\left[ \exp\left(\int^{T\wedge \tau_\varepsilon}_0 \mathcal{M}\abs{\nabla^2 u^R}^2(X^R_s(x))\,ds\right)\right],
\end{equation}

\begin{equation}\label{eq:max2}
    M_2:=\mathbb{E}\left[ \exp\left(\int^{T\wedge \tau_\varepsilon}_0 \mathcal{M}\norm{\nabla \sigma^R}^2(X^R_s(x))\,ds\right)\right],
\end{equation}
and
\begin{equation}\label{eq:max3}
    M_3:=\mathbb{E}\left[ \exp\left(\int^{T\wedge \tau_\varepsilon}_0 \lambda \mathcal{M}\abs{\nabla u^R}(X^R_s(x))\,ds\right)\right].
\end{equation}
Take $f=\mathcal{M}\abs{\nabla^2u^R}^2$ and $p=\frac{p_1}{2}$ in \eqref{exp-estimate}, then we have
\begin{equation}
     M_1\le e\cdot\exp\left(T\left[ \frac{p_1(p_1-2) C_2 ({(T\lambda^R)}^{\frac{d}{p_1}}+{(T\lambda^R)}^{\frac{d}{p_1}-1})\norm{\mathcal{M}\abs{\nabla^2u^R}^2}_{\frac{p_1}{2}}}{1-e^{-1}} \right]^{(1-\frac{d}{p_1})^{-1}}\right).
\end{equation}
{We can take $T\lambda^R>1$, then ${(T\lambda^R)}^{\frac{d}{p_1}-1}<{(T\lambda^R)}^{\frac{d}{p_1}}$. By Theorem \ref{Max-inequality} (ii) and \eqref{uR2p1}, we have
\begin{equation}\label{max-ineq-bR}
   \norm{\mathcal{M}\abs{\nabla^2u^R}^2}_{\frac{p_1}{2}}\lesssim \norm{\nabla^2 u^R}^2_{p_1}\lesssim \norm{b^R}^2_{p_1}.
\end{equation}
Therefore,
\begin{equation}\label{eA1}
    \begin{aligned}
         M_1\le &e\cdot\exp\left( {\widetilde{\mathbf{C}}}\left[(\lambda^R)^{\frac{d}{p_1}} \norm{b^R}^2_{p_1}\right]^{(1-\frac{d}{p_1})^{-1}}\right)\\
\le &e\cdot\exp\left( {\widetilde{\mathbf{C}}}\,[\lambda^R]^{(1-\frac{d}{p_1})^{-1}} \right),
    \end{aligned}
\end{equation}
where the second inequality is due to $(\mathbf{H^{b^R}})$ and \eqref{definition_of_lambda^R}.
}
Similarly, taking $f=\mathcal{M}\norm{\nabla \sigma^R}^2$ and $p=\frac{p_1}{2}$ 
in \eqref{exp-estimate}, we obtain
\begin{equation*}
    \begin{aligned}
          M_2\le& e\cdot\exp\left({\widetilde{\mathbf{C}}} \,\left[(\lambda^R)^{\frac{d}{p_1}} \norm{\nabla \sigma^R}^2_{p_1}\right]^{(1-\frac{d}{p_1})^{-1}} \right)\\
\le& e\cdot\exp\left( {\widetilde{\mathbf{C}}}\,[\lambda^R+(\lambda^R)^{\frac{d}{p_1}}]^{(1-\frac{d}{p_1})^{-1}} \right)\\
\le& e\cdot\exp\left( {\widetilde{\mathbf{C}}}\,[\lambda^R]^{(1-\frac{d}{p_1})^{-1}} \right).
    \end{aligned}
\end{equation*}
Take $f=\lambda^R_H \cdot  \mathcal{M}\abs{\nabla u^R}$ and $p=\infty$, we obtain
\begin{equation*}
    \begin{aligned}
         & M_3\le e\cdot\exp\left( {\widetilde{\mathbf{C}}}\cdot \lambda^R \right)\le e\cdot\exp\left( {\widetilde{\mathbf{C}}}\,[\lambda^R]^{(1-\frac{d}{p_1})^{-1}} \right).
    \end{aligned}
\end{equation*}
By Novikov's criterion, the process
\begin{equation*}
    \begin{aligned}
         & t \mapsto \exp\left( 2\int_0^{t\wedge \tau_\varepsilon} \mathbf{B}_s\,d\widetilde{W}_s-2\int_0^{t\wedge \tau_\varepsilon} \abs{\mathbf{B}_s}^2\,ds \right)=:M^\varepsilon_t
\end{aligned}
\end{equation*}
is a continuous exponential martingale. By H\"older's inequality, we obtain
\begin{equation*}
    \begin{aligned}
          \mathbb{E}\abs{Z^R_{t\wedge \tau_\varepsilon}}^\alpha\le &2^\alpha\abs{x-y}^\alpha \left( \mathbb{E}M^\varepsilon_t \right)^\frac{1}{2}\left( \mathbb{E}\left[ \exp\left( \int^{t\wedge \tau_\varepsilon}_0 \abs{\mathbf{B}_s}^2\,ds +2 \int^{t\wedge \tau_\varepsilon}_0 \abs{\mathbf{A}_s}\,ds\right) \right] \right)^\frac{1}{2}\\ 
\le &C(\alpha,e)\exp\left( {\widetilde{\mathbf{C}}}\,[\lambda^R]^{(1-\frac{d}{p_1})^{-1}} \right)\abs{x-y}^{\alpha}.
    \end{aligned}
\end{equation*}
Let $\varepsilon\downarrow 0$,  we have
\begin{equation*}
    \begin{aligned}
         & \mathbb{E}\left[\abs{Y^R_t(\Phi_R(x))-Y^R_t(\Phi_R(y))}^\alpha \right]\le &C(\alpha,e)\exp\left( {\widetilde{\mathbf{C}}}\,[\lambda^R]^{(1-\frac{d}{p_1})^{-1}} \right)\abs{x-y}^{\alpha}.
    \end{aligned}
\end{equation*}
Moreover, if $\alpha>0$, then 
\begin{equation}\label{alpha_g_0}
    \begin{aligned}
         \mathbb{E}\left[\abs{X_t^R(x)-X_t^R(y)}^\alpha\right]&=\mathbb{E}\left[\abs{\Phi^{-1}_R(Y^R_t(\Phi_R(x)))-\Phi^{-1}_R(Y^R_t(\Phi_R(y)))}^\alpha\right]\\
         & \le \norm{\nabla \Phi^{-1}_R}^\alpha_\infty\mathbb{E}[\abs{Z^R_t}^\alpha]\\
         &\le C(\alpha,e)\exp\left( {\widetilde{\mathbf{C}}}\,[\lambda^R]^{(1-\frac{d}{p_1})^{-1}} \right)\abs{x-y}^{\alpha}.
    \end{aligned}
\end{equation}
Notice  that
\begin{equation}
    \begin{aligned}
         & \abs{Y^R_t(\Phi_R(x))-Y^R_t(\Phi_R(y))}=\abs{\Phi_R(X^R_t(x))-\Phi_R(X^R_t(y))}\le 2\abs{X^R_t(x)-X^R_t(y)},
    \end{aligned}
\end{equation}
if $\alpha<0$, then
\begin{equation}\label{alpha_l_0}
    \begin{aligned}
         & \abs{X^R_t(x)-X^R_t(y)}^{\alpha}\\
\le &2^{-\alpha}\abs{Y^R_t(\Phi_R(x))-Y^R_t(\Phi_R(y))}^{\alpha}\\\le &C(\alpha,e)\exp\left( {\widetilde{\mathbf{C}}}\,[\lambda^R]^{(1-\frac{d}{p_1})^{-1}} \right)\abs{x-y}^{\alpha}.
    \end{aligned}
\end{equation}
Together, \eqref{alpha_g_0} and \eqref{alpha_l_0} imply \eqref{eq:XR-x-y}.\\
{Notice that 
\[
\Phi_R(\Phi_R^{-1}(x))=x,\ \Phi_R(x)=x+u^R(x),
\]
we have 
\[
\Phi_R^{-1}(x)+u^R(\Phi_R^{-1}(x))=x.
\]
Therefore, 
\begin{equation}\label{Phi0}
   \abs{\Phi_R(x)}\vee\abs{\Phi^{-1}_R(x)}\le \abs{x}+\norm{u^R}_\infty\le \abs{x}+\frac{1}{2}. 
\end{equation}
}
{By $X^R_s(x)=\Phi_R^{-1}(Y^R_s(\Phi_R(x)))$, \eqref{grad_bound} and \eqref{Phi0}, we have
\[
  \frac{1}{2}\left(1+\abs{Y^R_s(\Phi_R(x))}\right)\le 1+\abs{X^R_s(x)}\le  2\left(1+\abs{Y^R_s(\Phi_R(x))}\right).
\]
Combining the inequality 
\[
\frac{1}{2} (1+\abs{x})^2\le (1+\abs{x}^2)\le (1+\abs{x})^2,   
\]
we can obtain
\begin{equation}
    \begin{aligned}
         & \left( 1+\abs{X^R_s(x)}^2 \right)^\alpha\le C(\alpha)\left( 1+\abs{Y^R_s(\Phi_R(x))}^2 \right)^\alpha
    \end{aligned}
\end{equation}
where $C(\alpha)=8^\alpha \vee 8^{-\alpha}$. Therefore, we just need to consider the  estimate of $\mathbb{E}\left[ \left( 1+ \abs{Y^R_s(\Phi_R(x))}^2\right)^\alpha \right]$.

Applying It\^o formula to $\left(1+\abs{Y^R_s(\Phi_R(x))}^2\right)^\alpha$, we have
\begin{equation}
    \begin{aligned}
          (1+\abs{Y^R_t}^2)^\alpha &=(1+\abs{\Phi_R(x)}^2)^\alpha + 2\alpha\int^t_0 (1+\abs{Y^R_s}^2)^{\alpha -1}\langle Y^R_s, \tilde{\sigma}^R(Y^R_s)d\widetilde{W}_s\rangle \\
          & + 2\alpha\int^t_0 (1+\abs{Y^R_s}^2)^{\alpha -1}\langle \tilde{b}(Y^R_s), Y^R_s)\rangle\,ds\\ 
          & +\alpha \int^t_0 (1+\abs{Y^R_s}^2)^{\alpha-1}\norm{\sigma(Y^R_s)}^2\,ds\\
          & +2\alpha(\alpha-1)\int^t_0(1+\abs{Y^R_s}^2)^{\alpha-2}\abs{\tilde{\sigma}^R(Y^R_s)Y^R_s}^2\,ds.
    \end{aligned}
\end{equation}
By \eqref{observe} and \eqref{sup_Y^R}, we obtain
\[
\mathbb{E}\left[  (1+\abs{Y^R_t}^2)^\alpha \right] \le \tilde{\mathbf{C}} (1+\abs{x}^2)^\alpha +(\tilde{\mathbf{C}}\, \lambda^R+\tilde{\mathbf{C}} ) \int^t_0\mathbb{E}\left[  (1+\abs{Y^R_s}^2)^\alpha \right]\,ds.
\]
Using Gronwall's inequality, we proved  \eqref{1-plus-X}.
}

It is easy to see
\begin{equation*}
  \begin{aligned}
     & \mathbb{E}\left[ \sup_{0\le s\le t}\abs{X^R_s(x)}^p\right]\\
\le& \mathbb{E}\left[ \sup_{0\le s\le t}\abs{\Phi_R^{-1}(Y^R_s(\Phi_R(x)))}^p\right]\\
\le & \mathbb{E}\left[ \sup_{0\le s\le t}\abs{\Phi_R^{-1}(Y^R_s(\Phi_R(x)))-\Phi_R^{-1}(0)+\Phi_R^{-1}(0)}^p\right]\\
\le & C(p)\mathbb{E}\left[ \sup_{0\le s \le t} \abs{Y_s^R(\Phi_R(x))}^p \right] + C(p) \abs{\Phi_R^{-1}(0)}^p\\
\le &C(p) \mathbb{E}\left[ \sup_{0\le s \le t} \abs{Y_s^R(\Phi_R(x))}^p \right]+C(p),
  \end{aligned}
\end{equation*}
where the last inequality is due to $\norm{\nabla \Phi_R^{-1}}_{\infty}\le 2$ and $\Phi_R^{-1}(0)\le 1/2$. So, we only need to estimate $\mathbb{E}\left[ \sup_{0\le s\le t} \abs{Y^R_s(\Phi_R(x))}^p \right],\  p\ge 2$.\\
By the equation \eqref{eq:Zvonkin-SDE}, we have
\begin{equation}\label{sup_Y^R}
  \begin{aligned}
     & \mathbb{E}\left[ \sup_{0\le s\le t} \abs{Y^R_s}^p \right]\\
\le& C(p)\mathbb{E}\left[ \abs{\Phi_R(x)}^p +\sup_{0\le s\le t}\abs{\int^s_0 \tilde{b}^R(Y^R_r)\,dr}^p +\sup_{0\le s\le t}\abs{\int^s_0 \tilde{\sigma}^R(Y^R_r)\,d\widetilde{W}_r}^p \right]\\
:=&C(p)(I_1+I_2+I_3).
  \end{aligned}
\end{equation}
It is not hard to see
\begin{equation*}
    \begin{aligned}
         & I_1\le (x+\norm{u^R}_\infty)^p\le C(p)(1+\abs{x}^p),\\
& I_2 \le \mathbb{E}\left[ t^{p-1}\int^t_0\abs{\tilde{b}^R(Y^R_r)}^p\,dr \right]\le t^p\norm{\tilde{b}^R}^p_\infty\le \frac{1}{2^p}t^p\lambda^p,\\
&I_3\le \mathbb{E}\left[ \left(\int^t_0 \norm{\tilde{\sigma}^R(Y^R_r)}^2\,dr \right)^\frac{p}{2} \right]\le t^\frac{p}{2}\norm{\tilde{\sigma}^R}^p_\infty{\le t^\frac{p}{2}2^p\norm{\sigma^R}^p_\infty}.
    \end{aligned}
\end{equation*}
So, we obtained \eqref{X^R_one_poit_estimate}.\\
Notice that
$$\mathbb{E}[\sup_{0\le t\le T }\abs{\Phi^{-1}_R(Y^R_t(\Phi_R(x)))-\Phi^{-1}_R(Y^R_t(\Phi_R(y)))}^p]\le 2^p\mathbb{E}[\sup_{0\le t\le T}\abs{Y_t^R(\Phi_R(x))-Y_t^R(\Phi_R(y))}^p],$$
we only need to estimate $\mathbb{E}[\sup_{0\le t\le T}\abs{Z^R_t}^p]$. By \eqref{D_D_formula}, we have 
\begin{equation*}
    \begin{aligned}
          &\mathbb{E}[\sup_{0\le t\le T}\abs{Z^R_t}^p]\\
\le& \abs{\Phi_R(x)-\Phi_R(y)}^p\left( \mathbb{E}{\sup_{0\le t\le T} {M}^2_1(t)} \right)^\frac{1}{2}\left( \exp\left( 2\int^T_0 \abs{\mathbf{A}_s}\,ds\right) \right)^{\frac{1}{2}}\\
\le& \abs{\Phi_R(x)-\Phi_R(y)}^p\left( \mathbb{E}{{M}^2_1(T)} \right)^\frac{1}{2}\left( \exp\left( 2\int^T_0 \abs{\mathbf{A}_s}\,ds\right) \right)^{\frac{1}{2}}\\
\le& \abs{\Phi_R(x)-\Phi_R(y)}^p\left( \mathbb{E}{{M}_4(T)} \right)^\frac{1}{4}\left( \exp\left( 6\int^T_0 \abs{\mathbf{B}_s}^2\,ds\right) \right)^{\frac{1}{4}}\left( \exp\left( 2\int^T_0 \abs{\mathbf{A}_s}\,ds\right) \right)^{\frac{1}{2}}\\
\le& \widetilde{\mathbf{C}}\left(\exp\left(\widetilde{\mathbf{C}}\,(\lambda^R)^{\frac{p_1}{p_1-d}}\right)\right)\abs{x-y}^p,
    \end{aligned}
\end{equation*}
where 
\[
M_k(t):=\exp\left( k\int_0^{t} \mathbf{B}_s\,d\widetilde{W}_s-\frac{k^2}{2}\int_0^{t} \abs{\mathbf{B}_s}^2\,ds \right). 
\]
We proved \eqref{sup_xy}.

\end{proof}}
{ 
Let $D_t(x):=\sup_{0\le s\le t}\abs{X_s(x)}$, $\tau_R(x):=\inf\{t\ge 0, \abs{X_t(x)}>R\}$ and similarly, let $D^R_t(x):=\sup_{0\le s\le t}\abs{X^R_s(x)}$, $\tau^R_R(x):=\inf\{ t\ge 0, \abs{X^R_t(x)}>R \}$. It is easy to see 
$$\{D_t(x)\ge R\}=\{ \tau_R\le t \}, \{D^R_t(x)\ge R\}=\{ \tau^R_R\le t \}.$$
By the definitions of $b^R$ and $\sigma^R$, it is not hard to obtain
$$ \{\tau_R\le t \}\subset \{\tau^R_R\le t \}. $$
For all $x\in B(N)$, we have 
   \begin{equation*}
  \begin{aligned}
     \mathbb{P}(\tau_R\le t)&\le \mathbb{P}(\tau^R_R\le t)=\mathbb{P}(D^R_t(x)\ge R)\\
&\le \frac{\mathbb
{E}[\abs{D^R_t(x)}^n]}{R^n}\\
&\le \frac{\widetilde{\mathbf{C}}(1+\abs{x}^n+(\lambda^R)^n)}{R^n},
  \end{aligned}
\end{equation*}
where the second inequality is due to Markov's inequality, the last inequality is due to Lemma \ref{lm:XR-two-point}. By the definition of $\lambda^R$ in  \eqref{definition_of_lambda^R}, we can obtain ${(\lambda^R)^n}/{R^n}\rightarrow 0$ when $R\rightarrow \infty$.
Hence, we have $\tau_R\rightarrow \infty$ when $R\rightarrow \infty$.
On the other hand, by the definitions of $b^R$ and $\sigma^R$, we observe that if $D_t(x)<R$, then $X_t(x)=X^R_t(x)$ i.e. 
$X_{t}(x)=X^R_{t}(x)$ for all $t<\tau_R$. By  Theorem \ref{th:XR-exist},  SDE \eqref{eq:newsde} exists a unique strong solution. We can define $X_t(x)=X^R_t(x)$ for $t<\tau_R$. It is clear that $\{X_t(x)\}_{t\in[0,T]}$ is the unique strong solution of SDE \eqref{eq:b}.

By \eqref{X^R_one_poit_estimate} and definition of $\lambda^R$, for all $x\in B(N)$, we have
\begin{equation}\label{eq:sup_one_point}
  \begin{aligned}
     &\mathbb{E}[\sup_{0\le t\le T}\abs{X_t(x)}^p]\\
\le&\sum^\infty_{R=1}  \mathbb{E}\left[\abs{D^R_T(x)}^p\mathds{1}_{\{ R-1\le D_T(x)<R \}}\right]\\
\le&\sum^\infty_{R=2}  \mathbb{E}\left[\abs{D^R_T(x)}^p\mathds{1}_{\{ R-1\le D_T(x)<R \}}\right]+\mathbf{C}(N)\\
\le&\sum^\infty_{R=2}  \mathbb{E}\left[\abs{D^R_T(x)}^{2p} \right]^{\frac{1}{2}}  \left[\mathbb{P}(D^{R-1}_T(x)\ge R-1)\right]^{\frac{1}{2}}+\mathbf{C}(N)\\
\le& \sum^\infty_{R=2}  \mathbb{E}\left[\abs{D^R_T(x)}^{2p} \right]^{\frac{1}{2}}\cdot \frac{\mathbb{
E}[(D_t^{R-1}(x))^{2p}]^{\frac{1}{2}}}{(R-1)^p}+\mathbf{C}(N)\\
\le& \sum^\infty_{R=2} \frac{\mathbb{E}[(D^R_T(x))^{2p}]^\frac{1}{2}\cdot\mathbb{
E}[(D_T^{R-1}(x))^{2p}]^{\frac{1}{2}}}{(R-1)^p}+\mathbf{C}(N)\\
\le& \mathbf{C}(N).
  \end{aligned}
\end{equation}
where the last inequality  is due to  \eqref{X^R_one_poit_estimate} and  the definition  of $\lambda^R$.\\
For all $x,y\in B(N)$, we consider the following estimate
\begin{equation}\label{key_estimate}
  \begin{aligned}
     & \mathbb{E}\left[\sup_{0\le t\le T} \abs{X_t(x)-X_t(y)}^p \right]\\
=& \sum^\infty_{R=1}  \mathbb{E}\left[\sup_{0\le t\le T}\abs{X^R_t(x)-X^R_t(y)}^p\mathds{1}_{\{R-1\le D_T(x)\vee D_T(y)<R\}}\right]\\
         \le& \sum^\infty_{R=1}\left(\mathbb{E}\left[\sup_{0\le t\le T}\abs{X^R_t(x) - X^R_t(y)}^{2p}\right]\right)^{\frac{1}{2}}\mathbb{P}\Big(D_T(x)\vee D_T(y)\ge R-1\Big)^{\frac{1}{2}}\\
         \le& \sum^\infty_{R=1}\left(\mathbb{E}\left[\sup_{0\le t\le T}\abs{X^R_t(x) - X^R_t(y)}^{2p}\right]\right)^{\frac{1}{2}}\Big(\mathbb{P}(D_T(x)\ge R-1) + \mathbb{P}(D_T(y)\ge R-1) \Big)^{\frac{1}{2}}.\\ 
         \le& \sum^\infty_{R=1}\left(\mathbb{E}\left[\sup_{0\le t\le T}\abs{X^R_t(x) - X^R_t(y)}^{2p}\right]\right)^{\frac{1}{2}}\Big(\mathbb{P}(D^{R-1}_T(x)\ge R-1) + \mathbb{P}(D^{R-1}_T(y)\ge R-1) \Big)^{\frac{1}{2}}\\ 
        \le & \sum^\infty_{R=2}\left(\mathbb{E}\left[\sup_{0\le t\le T}\abs{X^R_t(x) - X^R_t(y)}^{2p}\right]\right)^{\frac{1}{2}}\left(\frac{\mathbb{E}[(D^{R-1}_T(x))^{2n}]}{(R-1)^{2n}}+\frac{\mathbb{E}[(D^{R-1}_T(y))^{2n}]}{(R-1)^{2n}}\right)^\frac{1}{2}+\mathbf{C}\abs{x-y}^p\\
        \le &\sum^\infty_{R=2} \widetilde{\mathbf{C}}\abs{x-y}^p{\left(\exp\left(\widetilde{\mathbf{C}}\,(\lambda^R)^{\frac{p_1}{p_1-d}}\right)\right)\frac{(1+\abs{x}^n)}{(R-1)^n}}+\sum^\infty_{R=2} \widetilde{\mathbf{C}}\abs{x-y}^p{\left(\exp\left(\widetilde{\mathbf{C}}\,(\lambda^R)^{\frac{p_1}{p_1-d}}\right)\right)\frac{(\lambda^R)^n}{(R-1)^n}}+\\
         & \sum^\infty_{R=2} \widetilde{\mathbf{C}}\abs{x-y}^p{\left(\exp\left(\widetilde{\mathbf{C}}\,(\lambda^R)^{\frac{p_1}{p_1-d}}\right)\right)\frac{(1+\abs{y}^n)}{(R-1)^n}}+\mathbf{C}\abs{x-y}^p\\
\le & \sum^\infty_{R=2} \widetilde{\mathbf{C}}\abs{x-y}^p{\left(\exp\left(2\widetilde{\mathbf{C}}\,(\lambda^R)^{\frac{p_1}{p_1-d}}\right)\right)\frac{(2+\abs{x}^n)}{(R-1)^n}}+\mathbf{C}\abs{x-y}^p\\
&+\sum^\infty_{R=2} \widetilde{\mathbf{C}}\abs{x-y}^p{\left(\exp\left(2\widetilde{\mathbf{C}}\,(\lambda^R)^{\frac{p_1}{p_1-d}}\right)\right)\frac{(2+\abs{y}^n)}{(R-1)^n}},
\end{aligned}
\end{equation}
where the last inequality we used the fact that we can find a constant ${C}(\widetilde{\mathbf C},p_1,d,n(\beta))$ such that  for all $\lambda^R\ge {C}(\widetilde{\mathbf C},p_1,d,n(\beta))$,
\begin{equation}\label{lambda^R_condition_1}
    \begin{aligned}
         & (\lambda^R)^n\le  \exp\left(\widetilde{\mathbf{C}}\,(\lambda^R)^{\frac{p_1}{p_1-d}}\right).
    \end{aligned}
\end{equation}
In fact, if let $\tilde{\beta}$ satisfy $(2C_2\tilde \beta)^{2(1-\frac{d}{p_1})^{-1}}= {C}(\widetilde{\mathbf C},p_1,d,n(\beta))$, then for all $R\ge 1$, $\lambda^R$ satisfy \eqref{lambda^R_condition_1}, where $n(\beta)$ be decided by \eqref{n(beta)}.\\
On the other hand, by the definitions of $\lambda^R$ and $I_b(R)$, we have
\begin{equation*}
    \begin{aligned}
         &  \mathbb{E}\left[ \sup_{0\le t\le T}\abs{X_t(x)-X_t(y)}^p \right]\\
\le &\sum^\infty_{R=2} \mathbf{C}(\beta,\tilde{\beta}) R^{\mathbf{C}(\beta)}\frac{(2+\abs{x}^n)}{(R-1)^n}  + \sum^\infty_{R=2} \mathbf{C}(\beta,\tilde{\beta}) R^{\mathbf{C}(\beta)}\frac{(2+\abs{y}^n)}{(R-1)^n}+\mathbf{C}\abs{x-y}^p.
\end{aligned}
\end{equation*}
Therefore, take $n$ satisfy
\begin{equation}\label{n(beta)}
    \begin{aligned}
         & \mathbf{C}(\beta)+1<n, 
    \end{aligned}
\end{equation}
we obtain
\begin{equation}\label{eq:sup_two_point}
    \begin{aligned}
          &  \mathbb{E}\left[\sup_{0\le t\le T} \abs{X_t(x)-X_t(y)}^p \right]\le \mathbf{C} \Big((1+\abs{x}^n)+(1+\abs{y}^n)\Big)\abs{x-y}^p.
    \end{aligned}
\end{equation}
By the Lemma $2.1$ in \cite{xie2016}, \eqref{eq:sup_one_point} and \eqref{eq:sup_two_point}, we proved Theorem \ref{main-theorem}(A).

Following the proof of Zhang \cite{zhang2011b}, it is not hard to prove for any bounded measurable function $f$ and $t\in [0,T]$, 
\begin{equation}\label{eq:R_continuous}
    \begin{aligned}
         & x\mapsto \mathbb{E}[f(X^R_t(x))] \text{\ is continuous.}
    \end{aligned}
\end{equation}
For any $x,y\in B(N)$, we have
\begin{equation}\label{strong_feller}
    \begin{aligned}
         & \abs{\mathbb{E}\left[ f(X_t(x)-f(X_t(y))) \right]}\\
\le&\abs{\mathbb{E}\left[ (f(X_t(x)-f(X_t(y))))\mathds{1}_{\{t\le \tau_R\}} \right]}+2\norm{f}_\infty\mathbb{P}(t>\tau_R)\\
\le& \abs{\mathbb{E}\left[ (f(X^R_t(x)-f(X^R_t(y))))\mathds{1}_{\{t\le \tau_R\}} \right]}+2\norm{f}_\infty\mathbb{P}(t>\tau_R)\\
\le& \abs{\mathbb{E}\left[ (f(X^R_t(x)-f(X^R_t(y)))) \right]}+4\norm{f}_\infty\mathbb{P}(t>\tau_R)\\
    \end{aligned}
\end{equation}
Together, \eqref{strong_feller}, \eqref{eq:R_continuous} and $\tau_R\rightarrow \infty$ when $R\rightarrow \infty$ imply Theorem \ref{main-theorem}(B).
}

{\begin{lemma}\label{th:two-point-estimate-final}
    Under $\mathbf{(H^b)}$, $\mathbf{(H^\sigma_1)}$ and  $\mathbf{(H^\sigma_2)}$, let $\{X_t(x)\}_{t\in [0,T]}$ and $\{X_t(y)\}_{t\in[0,T]}$  are two solutions of  SDE \eqref{eq:b} with initial conditions $X_0(x)=x$ and $X_0(y)=y$ respectively, then  for all $0\le t\le T$, $\alpha\in \mathbb{R}$ and $x,y\in B(N)$, we have
    \begin{equation}\label{eq:X-x-y}
        \begin{aligned}
             & \mathbb{E}[\abs{X_t(x)-X_t(y)}^\alpha]\le \mathbf{C}(N) \abs{x-y}^\alpha,
        \end{aligned}
    \end{equation}

{\begin{equation}\label{1-plus-x^2}
    \begin{aligned}
         & \mathbb{E}\left[ \left( 1+\abs{X_t(x)}^2 \right)^\alpha \right]\le \mathbf{C}(N)\left( 1+\abs{x}^2  \right)^\alpha,
    \end{aligned}
    \end{equation}}
and for all $p\ge 2$,
\begin{equation}\label{t-s}
    \begin{aligned}
         & \mathbb{E}[\abs{X_t(x)-X_s(x)}^p]\le \mathbf{C}(N)\abs{t-s}^{\frac{p}{2}}.
    \end{aligned}
\end{equation}
\end{lemma}
\begin{proof}
   Set $D_t(x):=\sup_{0\le s\le t}\abs{X_t(x)}$ and $D_t(y):=\sup_{0\le s\le t}\abs{X_t(y)}$. It is easy to see
if $D_t(x)< R$ and $D_t(y)< R$, then 
$
X_t(x) = X^R_t(x), X_t(y) = X_t^R(y).
$ 
Moreover, by Lemma \ref{lm:XR-two-point}, similar to \eqref{key_estimate}, for all $t\in [0,T]$ and $x,y\in B(N)$, we have 
\begin{equation}\label{eq:deposite}
    \begin{aligned}
        &\mathbb{E}[\abs{X_t(x)-X_t(y)}^\alpha] \\
        =& \sum^\infty_{R=1}  \mathbb{E}\left[\abs{X^R_t(x)-X^R_t(y)}^\alpha\mathds{1}_{\{R-1\le D_T(x)\vee D_T(y)<R\}}\right]\\
         \le& \sum^\infty_{R=1}\left(\mathbb{E}\left[\abs{X^R_t(x) - X^R_t(y)}^{2\alpha}\right]\right)^{\frac{1}{2}}\mathbb{P}\Big(D_T(x)\vee D_T(y)\ge R-1\Big)^{\frac{1}{2}}\\
         \le& \sum^\infty_{R=1}\left(\mathbb{E}\left[\abs{X^R_t(x) - X^R_t(y)}^{2\alpha}\right]\right)^{\frac{1}{2}}\Big(\mathbb{P}(D_T(x)\ge R-1) + \mathbb{P}(D_T(y)\ge R-1) \Big)^{\frac{1}{2}}\\ 
\le & \sum^\infty_{R=2}\left(\mathbb{E}\left[\abs{X^R_t(x) - X^R_t(y)}^{2\alpha}\right]\right)^{\frac{1}{2}}\left(\frac{\mathbb{E}[(D^{R-1}_T(x))^{2n}]}{(R-1)^{2n}}+\frac{\mathbb{E}[(D^{R-1}_T(y))^{2n}]}{(R-1)^{2n}}\right)^\frac{1}{2}+\mathbf{C}\abs{x-y}^\alpha\\
\le &{\mathbf{C}\,(1+\abs{x}^n+\abs{y}^n)\abs{x-y}^\alpha}\\
\le& \mathbf{C}(N)\abs{x-y}^\alpha,
    \end{aligned}
\end{equation}
{and 
\begin{equation*}
    \begin{aligned}
         &  \mathbb{E}\left[ \left( 1+\abs{X_t(x)}^2 \right)^\alpha \right]\\
=&  \sum^\infty_{R=1}  \mathbb{E}\left[\left( 1+\abs{X^R_t(x)}^2 \right)^\alpha\mathds{1}_{\{R-1\le D_T(x)<R\}}\right]\\
\le& \sum^\infty_{R=2} \left( \mathbb{E}\left[  \left( 1+\abs{X^R_t(x)}^2 \right)^{2\alpha}\right]  \right)^\frac{1}{2} \left( \frac{\mathbb{E}[(D^{R-1}_T(x))^{2n}]}{(R-1)^{2n}} \right)^\frac{1}{2} +\mathbf{C}(1+\abs{x}^2)^\alpha\\
\le & \mathbf{C}\,\big(1+\abs{x}^n\big)\,\big(1+\abs{x}^2\big)^\alpha\\
\le & \mathbf{C}(N)(1+\abs{x}^2)^\alpha.
    \end{aligned}
\end{equation*}
}
On the other hand, it is not hard to obtain
\begin{equation*}
    \begin{aligned}
         &\mathbb{E}[\abs{X^R_t(x)-X^R_s(x)}^p]\\
         \le& C(p) \mathbb{E}[\abs{Y^R_t(\Phi_R(x))-Y^R_s(\Phi_R(x))}^p]\\
        \le &\mathbf{C}(T)\big(1+(\lambda^R)^p\big)\abs{t-s}^\frac{p}{2},
    \end{aligned}
\end{equation*}
where the last inequality is due to 
\begin{equation*}
    \begin{aligned}
         & \mathbb{E}\left[\abs{\int^t_s \tilde{b}^R(Y^R_r)\,dr}^p\right]\le ||\tilde{b}^R||^p_\infty\abs{t-s}^p,
    \end{aligned}
\end{equation*}
and 
\begin{equation*}
    \begin{aligned}
         & \mathbb{E}\left[\abs{\int^t_s \tilde{\sigma}^R(Y^R_r)\,d\widetilde{W}_r}^p\right]\le    ||\tilde{\sigma}^R||^p_\infty\abs{t-s}^\frac{p}{2}.
    \end{aligned}
\end{equation*}
Moreover, for all $t,s\in [0,T]$ and $x\in B(N)$, we have
\begin{equation*}
    \begin{aligned}
         & \mathbb{E}[\abs{X_t(x)-X_s(x)}^p]\\
         =& \sum^\infty_{R=1}  \mathbb{E}\left[\abs{X^R_t(x)-X^R_s(x)}^p\mathds{1}_{\{R-1\le D_T(x)<R\}}\right]\\
         \le& \sum^\infty_{R=2}\left( \mathbb{E}\left[ \abs{X^R_t(x)-X^R_s(x)}\right]^{2p}\right)^\frac{1}{2}\left(\frac{\mathbb{E}[(D^{R-1}_T(x))^{2p}]}{(R-1)^{2p}} \right)^\frac{1}{2}+\mathbf{C}\abs{t-s}^\frac{p}{2}\\
         \le& \sum^\infty_{R=2} \mathbf{C}(T)\frac{\bigl( 1+\abs{x}^p+(\lambda^R)^p \bigr)^2}{(R-1)^{p}}\abs{t-s}^\frac{p}{2} +\mathbf{C}\abs{t-s}^\frac{p}{2}\\
         \le& \mathbf{C} (1+\abs{x}^{2p})\abs{t-s}^\frac{p}{2}\\
        \le& \mathbf{C}(N)\abs{t-s}^\frac{p}{2}.
    \end{aligned}
\end{equation*}
We completed the proof.
\end{proof}
}

By the Lemma \ref{th:two-point-estimate-final}, for all $p\ge2$, $t,s\in[0,T]$ and $x,y\in B(N)$, we have
\begin{equation}\label{Lp-x-t}
    \mathbb{E}\left[\abs{X_t(x)-X_s(y)}^p\right]\le \mathbf{C}(N)\left( \abs{x-y}^p+\abs{t-s}^{\frac{p}{2}} \right).
\end{equation} 
By Kolmogorov's lemma, we can obtain for any $N\in \mathbb{N}$, there exists a $\mathbb{P}$-null set $\Xi_N $ such that for any $\omega\notin \Xi_N$, $X_{\cdot}(\omega,\cdot):[0,T]\times B(N)\rightarrow \mathbb{R}^d$ is continuous.
If we set $\Xi:=\cup_{N=1}^\infty\Xi_N$, then $\mathbb{P}(\Xi)=0$ and 
$$
X_{\cdot}(\omega,\cdot):[0,T]\times\mathbb{R}^d\rightarrow \mathbb{R}^d \text{\ is continuous}, \ \ \forall \omega \notin \Xi.
$$
Similar to  the standard argument (cf. \cite{kunita1990}), the proof of for any $t\in [0,T]$, almost all $\omega$, the maps $x\mapsto X_{t}(\omega,x)$ is one-to-one  due to \eqref{eq:X-x-y} and \eqref{t-s}. For the reader's convenience, we give the details of one-to-one property.

For $x\neq y \in \mathbb{R}^d$, set
\begin{equation*}
    \begin{aligned}
         & \mathscr{R}(t,x,y):=\frac{1}{\abs{X_t(x)-X_t(y)}},
    \end{aligned}
\end{equation*}
then 
\begin{equation*}
    \begin{aligned}
         & \abs{\mathscr{R}(t,x,y)-\mathscr{R}(s,x',y')}\\
   \le &\frac{\abs{X_t(x)-X_t(y)-X_s(x')+X_s(y')}}{\abs{X_t(x)-X_t(y)}\abs{X_s(x')-X_s(y')}}\\ 
\le&\frac{\abs{X_t(x)-X_t(x')}+\abs{X_t(x')-X_s(x')}+\abs{X_t(y)-X_t(y')}+\abs{X_t(y')-X_s(y')}}{\abs{X_t(x)-X_t(y)}   \abs{X_s(x')-X_s(y')}}.
    \end{aligned}
\end{equation*}
By H\"older inequality, we have
\begin{equation*}
    \begin{aligned}
         & \mathbb{E}\abs{\mathscr{R}(t,x,y)-\mathscr{R}(s,x',y')}^p\\ 
      \le&\mathbf{C}\cdot\mathbb{E}\left[ \abs{X_t(x)-X_t(x')}^{2p} +\abs{X_t(x')-X_s(x')}^{2p} +\abs{X_t(y)-X_t(y')}^{2p}+\abs{X_t(y')-X_s(y')}^{2p}\right]^{\frac{1}{2}}\cdot\\
&\ \ \mathbb{E}\left[ \abs{X_t(x)-X_t(y)}^{-4p} \right]^{\frac{1}{4}} \mathbb{E}\left[ \abs{X_s(x')-X_s(y')}^{-4p} \right]^{\frac{1}{4}}.
    \end{aligned}
\end{equation*}
Moreover, for all $x,y,x',y'\in B(N)$ and $\abs{x-y}\wedge \abs{x'-y'}>\varepsilon$, we obtain
\begin{equation*}
    \begin{aligned}
         & \mathbb{E}\abs{\mathscr{R}(t,x,y)-\mathscr{R}(s,x',y')}^p\\ 
\le&\mathbf{C}(N)\left( \abs{x-x'}^p+\abs{t-s}^{\frac{p}{2}}+\abs{y-y'}^{p} + \abs{t-s}^{\frac{p}{2}}\right)\varepsilon^{-2p}.
    \end{aligned}
\end{equation*}
Choose $p>4(d+1)$, by Kolmogorov's lemma,  there exists a $\mathbb{P}$-null set $\Xi_{k,N}$ such that for all $\omega\notin \Xi_{k,N}$,  the mapping $(t,x,y)\mapsto \mathscr{R}(t,x,y)$ is continuous on 
$$\{(t,x,y)\in [0,T]\times B(N)\times B(N):\abs{x-y}>\frac{1}{k}\}\quad \forall\,k\in \mathbb{N}_+.$$ Set $\Xi:=\cup^\infty_{k,N=1}, \Xi_{k,N}$, then for any $\omega\notin \Xi$, the mapping $(t,x,y)\mapsto \mathscr{R}(t,x,y)$ is continuous on 
$$\{(t,x,y)\in [0,T]\times \mathbb{R}^d\times \mathbb{R}^d:x\neq y\}.$$ We proved  one-to-one property.
\end{proof}

\section{Appendix}
\begin{proof}
{\bf The Proof of Theorem \ref{studied-pde-theorem}:}\ {\bf Step (i)}  Suppose $\sigma^R(x)$  does not depend on $x$, Krylov proved the estimate \eqref{krylov-lemma} in \cite[Page $109$]{krylov2008}. Therefore, If $\sigma^R(x)\equiv\sigma^R(x_0)$, then
\begin{equation*}
    \begin{aligned}
         & \norm{(\lambda - L^{\sigma^R(x_0)})^{-1}f}_{2,p}\le C_0 \norm{f}_p.
    \end{aligned}
\end{equation*}

{\bf Step (ii)} Suppose for some $x_0\in\mathbb{R}^d$ 
\begin{equation}\label{sigma-differential-for-fix-x_0}
    \norm{\sigma^R(x)-\sigma^R(x_0)}\le \frac{1}{2\tilde{\delta}^{-\frac{1}{2}} C_0},
\end{equation}
we consider the following equation
\begin{equation}
    L^{\sigma^R(x_0)} u- \lambda u + g=0,
\end{equation}
where $g:=L^{\sigma^R(x)}- L^{\sigma^R(x_0)} + f$.
By \eqref{sigma-differential-for-fix-x_0} and the definition of $L^{\sigma^R(x)}$, we obtain
$$
\norm{g}_p \le \frac{1}{2C_0}\norm{u_{xx}}_p + \norm{f}_p.
$$
Hence, by {\bf Step (i)}, we have
\begin{equation}
  \norm{u_{xx}}_p\le C_0\norm{g}_p\le\frac{1}{2}\norm{u_{xx}}_p + C_0\norm{f}_p,
\end{equation}   
i.e.
\begin{equation}
    \norm{u_{xx}}_p\le 2 C_0\norm{f}_p.
\end{equation}

{\bf Step (iii)}
Define a smooth cut-off function as follows
\begin{equation}
    \zeta(x)=\begin{cases}
1,\ \ \abs{x}\le 1,\\
\in [0,1],\ \ 1<x<2,\\
0\ \ \abs{x}\ge 2.
    \end{cases}
\end{equation}
Fix a small constant $\varepsilon$ which will be determined below.\\
For fixed $z\in \mathbb{R}^d$, let
\[
\zeta^\varepsilon_z(x) := \zeta(\frac{x-z}{\varepsilon}).   
\]
It is easy to check that
\begin{equation}\label{nabla-zeta}
    \int_{\mathbb{R}^d}\abs{\nabla^j_x \zeta^\varepsilon_z(x)}^p\,dz=\varepsilon^{d-jp}\int_{\mathbb{R}^d}\abs{\nabla^j \zeta(z)}^p\,dz >0, \quad j= 0,1,2.
\end{equation}
Multiply both side of \eqref{studyed-pde} by $\zeta^\varepsilon_z(x)$, we have

\begin{equation}
    L^{\sigma^R(x)}(u\zeta^\varepsilon_z) -\lambda (u\zeta^\varepsilon_z) +g^\varepsilon_z =0,
\end{equation}
where $g^\varepsilon_z:=(L^{\sigma^R(x)} u)\zeta^\varepsilon_z - L^{\sigma^R(x)}(u\zeta^\varepsilon_z) - f\zeta^\varepsilon_z$.\\
Let
$$
\hat{\sigma}^R(x):=\sigma^R((x-z)\zeta^{2\varepsilon}_z(x) + z).
$$
It is easy to obtain 
$$
L^{\sigma^R(x)}(u\zeta^{\varepsilon}_z) = L^{\hat{\sigma}^R(x)}(u\zeta^{\varepsilon}_z),
$$
since $\zeta^{2\varepsilon}_z(x) = 1 $ for $\abs{x-z}\le 2\varepsilon$ and $\zeta^\varepsilon_z(x) = 0$ for $\abs{x-z}>2\varepsilon$.\\
{By \eqref{holder_continous}} and the definition of $g^\varepsilon_z$ we have
$$
\norm{\hat{\sigma}^R(x)-\hat{\sigma}^R(z)}\le \tilde{\delta}^{-\frac{1}{2}} \abs{(x-z)\zeta^{2\varepsilon}_z}^\varpi\le \tilde{\delta}^{-\frac{1}{2}}\abs{4\varepsilon}^\varpi,
$$
and 

$$
\norm{g^\varepsilon_z}_p \le \norm{f\zeta^\varepsilon_z}_p + \tilde{\delta}^{-1} \norm{\abs{u_x}\abs{(\zeta^\varepsilon_z)_x}}_p + \tilde{\delta}^{-1} \norm{\abs{u}\abs{(\zeta^\varepsilon_z)_{xx}}}_p.
$$
By {\bf Step (ii)}, if 

$$
L^{\sigma^R(x)} u -\lambda u + f = 0,\quad    \norm{\sigma^R(x)-\sigma^R(x_0)}\le \frac{1}{2\tilde{\delta}^{-\frac{1}{2}} C_0},
$$
then 
$$
\norm{u_{xx}}_p\le 2C_0 \norm{f}_p.
$$
Now, we consider the following equation:
\begin{equation}
    L^{\hat{\sigma}^R(x)}(u \zeta^\varepsilon_z) -\lambda (u\zeta^\varepsilon_z) = g^\varepsilon_z
\end{equation}
and take $\varepsilon$ be small enough so that
$$
\norm{\hat{\sigma}^R(x)-\hat{\sigma}^R(z)}\le \tilde{\delta}^{-\frac{1}{2}}\abs{4\varepsilon}^\varpi\le \frac{1}{2\tilde{\delta}^{-\frac{1}{2}} C_0},
$$
then 
\begin{equation}\label{u_xi_xx}
 \norm{(u\zeta^\varepsilon_z)_{xx}}_p\le 2C_0\norm{g^\varepsilon_z}_p\le 2C_0\left( \norm{f\zeta^\varepsilon_z}_p + \tilde{\delta}^{-1} \norm{\abs{u_x}\abs{(\zeta^\varepsilon_z)_x}}_p + \tilde{\delta}^{-1} \norm{\abs{u}\abs{(\zeta^\varepsilon_z)_{xx}}}_p \right).   
\end{equation}
According to Fubini's theorem, \eqref{nabla-zeta} and \eqref{u_xi_xx}, it is easy to check
\begin{equation}
    \begin{aligned}
         & \int_{\mathbb{R}^d}\int_{\mathbb{R}^d} \abs{(u\zeta^\varepsilon_z)_{xx}}^p\,dx\,dz
         \le C(p,\varepsilon,\tilde{\delta}^{-1},C_0)\left( \norm{u_x}_p^p + \norm{u}_p^p +\norm{f}_p^p \right).
    \end{aligned}
\end{equation}
Moreover, we have 
\begin{equation}
    \begin{aligned}
         \norm{u_{xx}}_p^p\lesssim & \int_{\mathbb{R}^d}\norm{(u)_{xx}\cdot \zeta^\varepsilon_z}_p^p\,dz\\
         \lesssim &\int_{\mathbb{R}^d}\norm{(u\zeta^\varepsilon_z)_{xx}-(u)_x(\zeta^\varepsilon_z)_x-u(\xi_z^\varepsilon)_{xx}}_p^p\,dz\\
         \le & C(p,\varepsilon,\tilde{\delta}^{-1},C_0)\left( \norm{u_x}_p^p + \norm{u}_p^p +\norm{f}_p^p \right)\\
         \le & \frac{1}{2}\norm{u_{xx}}_p^p + C(p,\varepsilon,\tilde{\delta}^{-1},C_0)(\norm{u}_p^p+\norm{f}_p^p)
    \end{aligned}
\end{equation}
where the third inequality is due to \eqref{nabla-zeta} and \eqref{u_xi_xx} and the last inequality is due to 
\begin{equation}\label{ux-le-uxx+u}
    \norm{u_x}_p\le C(\norm{u_{xx}}_p+\norm{u}_p).
\end{equation}
and Young's inequality.
Therefore, we proved
$$
\norm{u_{xx}}_p\le C(p,\varepsilon,\tilde{\delta}^{-1},C_0)(\norm{u}_p+\norm{f}_p).
$$
Since $\lambda u=L^{\sigma^R(x)} u -f$, we have

\begin{equation}
    \begin{aligned}
         \lambda \norm{u}_p\le& \left(\norm{L^{\sigma^R(x)} u}_p+\norm{f}_p\right)\\
         \le&  C(d,\varpi,\tilde{\delta},p)\left( \norm{u}_p+\norm{f}_p \right).
    \end{aligned}
\end{equation}
Hence, we obtain
\begin{equation}
    \begin{aligned}
         & \norm{u_{xx}}_p +\lambda\norm{u}_p \le C(d,\varpi,\tilde{\delta},p)\left( \norm{u}_p+\norm{f}_p \right).
    \end{aligned}
\end{equation}
Notice that $\lambda >(C(d,\varpi,\tilde{\delta},p)+1)$, we obtain
\begin{equation}\label{uxx+u}
    \begin{aligned}
         & \norm{u_{xx}}_p +\norm{u}_p \le C(d,\varpi,\tilde{\delta},p)\norm{f}_p,
    \end{aligned}
\end{equation}
Combine \eqref{uxx+u} with \eqref{ux-le-uxx+u},
we get
$$\norm{u}_{2,p}\le C_1(d,\varpi,\tilde{\delta},p)\norm{f}_p. $$

{\bf Step (iv)} Set
$$
\mathcal{T}_{t}f(x):= \int_{\mathbb{R}^d} f(y)\rho(t,x,y)\,dy,
$$
where $\rho(t,x,y)$ is the fundamental solution of the operator $\partial_t - L^{\sigma^R(x)}$.  It is well-known that
\begin{equation}\label{heat_kernel}
    \abs{\nabla^j_x\rho(t,x,y)}\le C_j(\varpi,\tilde{\delta},d)t^{-j/2}(2t)^{-d/2}e^{-k_j(\varpi,\tilde{\delta},d)\abs{x-y}^2/(2t)}.
\end{equation}
By \cite[Lemma $3.4$]{zhang2016}, for any $p,p'\in (1,\infty)$ and $\alpha\in [0,2)$, there exists a constant  
$C=C(d,\varpi,\tilde{\delta},p,\alpha,p')$ such that for any $f\in L^p(\mathbb{R}^d)$,
\begin{equation}\label{mathcal-T-alpha-p'-estimate}
    \norm{\mathcal{T}_{t}f}_{\alpha,p'}\le C t^{(-\frac{\alpha}{2}-\frac{d}{2p}+\frac{d}{2p'})} \norm{f}_p.
\end{equation}
Let $f\in W^{2,p}(\mathbb{R}^d)$ and
\begin{equation}
    u(x):=\int^\infty_0 e^{-\lambda t}\,\mathcal{T}_{t}f(x)\,dt.
\end{equation}
By \eqref{heat_kernel} and the definition of $\mathcal{T}_{t}$, it is easy to check $u\in W^{2,p}(\mathbb{R}^d)$ and $u$ satisfies \eqref{studyed-pde}. Indeed,
\begin{equation}
    \begin{aligned}
          L^{\sigma^R(x)} u(x)&=\int^\infty_0 e^{-\lambda t}\int_{\mathbb{R}^d} f(y) L^{\sigma^R(x)}\rho(t,x,y) \,dy\,dt\\
          &=\int^\infty_0 e^{-\lambda t}\int_{\mathbb{R}^d} f(y) \partial_t\rho(t,x,y) \,dy\,dt\\
          &= \int_{\mathbb{R}^d} f(y) \left({\left.e^{-\lambda t}\rho(t,x,y)\right|^\infty_0 }+\lambda \int^\infty_0 e^{-\lambda t}\rho(t,x,y)\,dt\right)\,dy \\
          &= f(x) +\lambda u(x).
    \end{aligned}
\end{equation}
By Jensen's inequality, we obtain
\begin{equation}
    \begin{aligned}
          \abs{\Delta^{\frac{\alpha}{2}} u}^{p'}&=\abs{\int^\infty_0e^{-\lambda t}\Delta^{\frac{\alpha}{2}}\mathcal{T}_tf(x)\,dt}^{p'}\\
          &\le \left(\frac{1}{\lambda}\right)^{p'}\left( \int^\infty_0 \lambda e^{-\lambda t}\abs{\Delta^{\frac{\alpha}{2}}\mathcal{T}_t f(x)}^{p'}\,dt \right)
    \end{aligned}
\end{equation}
and 
\begin{equation}
    \abs{u}^{p'}\le \left(\frac{1}{\lambda}\right)^{p'}\left(\int^\infty_0 \lambda e^{-\lambda t} \abs{\mathcal{T}_t f(x)}^{p'}\,dt\right).
\end{equation}
By Fubini's theorem, we have
\begin{equation}\label{alpha/2}
    \begin{aligned}
          \norm{\Delta^{\frac{\alpha}{2}} u}^{p'}_{p'}
          \le \left(\frac{1}{\lambda}\right)^{p'}\left( \int^\infty_0 \lambda e^{-\lambda t}\norm{\Delta^{\frac{\alpha}{2}}\mathcal{T}_t f(x)}^{p'}_{p'}\,dt \right),
    \end{aligned}
\end{equation}
and 
\begin{equation}\label{alpha=0}
    \norm{u}^{p'}_{p'}\le \left(\frac{1}{\lambda}\right)^{p'}\left(\int^\infty_0 \lambda e^{-\lambda t} \norm{\mathcal{T}_t f(x)}^{p'}_{p'}\,dt\right).
\end{equation}
Moreover, by \eqref{frac-norm-equ}, \eqref{mathcal-T-alpha-p'-estimate},\eqref{alpha/2} and \eqref{alpha=0}, if $(\frac{d}{p}+\alpha-\frac{d}{p'})/2<\frac{1}{p'}\le 1$, then
\begin{equation}
    \begin{aligned}
         \norm{u}^{p'}_{\alpha,p'}&\lesssim \norm{f}_p^{p'} \left(\frac{1}{\lambda}\right)^{p'}\lambda\int^\infty_0 e^{-\lambda t}\,t^{(-\frac{\alpha}{2}-\frac{d}{2p}+\frac{d}{2p'})p'}\,dt\\
        &\le \norm{f}_p^{p'}\lambda^{-p'}\frac{1}{\lambda^{(-\frac{\alpha}{2}-\frac{d}{2p}+\frac{d}{2p'})p'}}\\
        &= \norm{f}_p^{p'}\lambda^{p'(\alpha-2+\frac{d}{p}-\frac{d}{p'})/2},
    \end{aligned}
\end{equation}
where the second inequality is due to Laplace transformation. 

{\bf Step (v)}  In this step, we will use weak convergence argument to prove the existence of  \eqref{studyed-pde}.  Let $\varphi$ be a nonnegative smooth function in $\mathbb{R}^d$ which satisfies $\int_{\mathbb{R}^d} \varphi(x)\,dx=1$ and support in $\{ x\in\mathbb{R}^d:\abs{x}\le 1\}$. Let
\begin{equation}
    \begin{aligned}
         & \varphi_n(x):=n^d\varphi(nx),\quad \sigma_n:=\sigma*\varphi_n,\quad f_n:=f*\varphi_n,
    \end{aligned}
\end{equation}
where $*$ denotes the convolution.\\
Denote $u_n$ be the solution of 

\begin{equation}
    L^{\sigma^R_n(x)} u_n -\lambda u_n= f_n.
\end{equation}
By the {\bf Step (iii)} and {\bf Step (iv)}, we have
\begin{equation}
    \norm{u_n}_{2,p}\le C_1\norm{f}_p
\end{equation}
and
\begin{equation}
    \norm{u_n}_{\alpha,p'}\le C_2 \lambda^{(\alpha-2+\frac{d}{p}-\frac{d}{p'})/2}\norm{f}_p.
\end{equation}
Since $W^{2,p}(\mathbb{R}^d)$ be weak compactness, we can find a subsequence still denoted by $u_n$ and $u\in W^{2,p}(\mathbb{R}^d)$ such that $u_n\rightharpoonup u$ in $W^{2,p}(\mathbb{R}^d)$.

For any test function $\phi\in C^\infty_0(\mathbb{R}^d)$, we have
\begin{equation}
    \begin{aligned}
          &\int_{\mathbb{R}^d} \left( L^{\sigma_m(x)} u_n - L^{\sigma(x)} u_n \right)\phi\,dx\\
\le& C_\phi\norm{\sigma_m-\sigma}_\infty \norm{(u_n)_{xx}}_p\\
\le& C_\phi\norm{\sigma_m-\sigma}_\infty \norm{f}_p
\rightarrow 0 \text{\quad $(m\rightarrow 0)$\quad uniformly in $n$},
    \end{aligned}
\end{equation}
and for fixed $m$

\begin{equation}
    \int_{\mathbb{R}^d} \left( L^{\sigma_m(x)} u_n - L^{\sigma_m(x)} u \right)\phi\,dx\rightarrow 0,\quad \text{as $n\rightarrow \infty$}.
\end{equation}
Hence, we obtain
\begin{equation}
    \int_{\mathbb{R}^d} \left( L^{\sigma_n(x)} u_n - L^{\sigma(x)} u \right)\phi\,dx\rightarrow 0,\quad \text{as $n\rightarrow \infty$}.
\end{equation}
Notice that
\begin{equation}
    \langle L^{\sigma_n(x)} u_n,\phi\rangle -\langle\lambda u_n,\phi\rangle= \langle f_n,\phi\rangle.
\end{equation}
Take $n\rightarrow \infty$, we obtain
\begin{equation}
    \langle L^{\sigma(x)} u,\phi\rangle -\langle\lambda u,\phi\rangle= \langle f,\phi\rangle.
\end{equation}
On the other hand, let $p_*:=\frac{p'}{p'-1}$ and keep in mind  $u_n\rightharpoonup u$ in $W^{2,p}(\mathbb{R}^d)$, we have
\begin{equation}
    \begin{aligned}
        \norm{u}_{\alpha,p'}=\norm{\left(I-\Delta^\frac{\alpha}{2}\right)u}_{p'}&=\sup_{\phi\in C^\infty_0(\mathbb{R}^d);\norm{\phi}_{p_*}\le 1}\abs{\int_{\mathbb{R}^d}\left\langle \left( I-\Delta^\frac{\alpha}{2}\right) u(x),\phi(x) \right\rangle\,dx}\\
        &= \sup_{\phi\in C^\infty_0(\mathbb{R}^d);\norm{\phi}_{p_*}\le 1}\lim_{n\rightarrow\infty}\abs{\int_{\mathbb{R}^d}\left\langle  u_n(x),\left( I-\Delta^\frac{\alpha}{2}\right)\phi(x)\right\rangle\,dx}\\ 
        &= \sup_{\phi\in C^\infty_0(\mathbb{R}^d);\norm{\phi}_{p_*}\le 1}\lim_{n\rightarrow\infty}\abs{\int_{\mathbb{R}^d}\left\langle  \left( I-\Delta^\frac{\alpha}{2}\right)u_n(x),\phi(x)\right\rangle\,dx}\\ 
        &\le \sup_n\sup_{\phi\in C^\infty_0(\mathbb{R}^d);\norm{\phi}_{p_*}\le 1}\norm{\left(I-\Delta^\frac{\alpha}{2}\right)u_n}_{p'}\\
 &=\sup_n \norm{u_n}_{\alpha,p'}\le C_2 \lambda^{(\alpha-2+\frac{d}{p}-\frac{d}{p'})/2}\norm{f}_p.
    \end{aligned}
\end{equation}
We completed the proof.
\end{proof}

\section*{Acknowledgement}
The author is greatly indebted to Professor Xin Chen for many useful discussions and for the guidance over the past years.
\newpage



\noindent
\begin{thebibliography}{10}

\bibitem{chen2014}
X.~Chen and X.-M. Li, Strong completeness for a class of stochastic
  differential equations with irregular coefficients, \emph{Electron. J.
  Probab.} \textbf{19} (2014), no. 91, 34.

\bibitem{Elworthy1994}
K.~D. Elworthy and X.-M. Li, Formulae for the derivatives of heat semigroups,
  \emph{J. Funct. Anal.} \textbf{125} (1994), no.~1, 252--286.

\bibitem{fang2007}
S.~Fang, P.~Imkeller and T.~Zhang, Global flows for stochastic differential
  equations without global {Lipschitz} conditions, \emph{Ann. Probab.}
  \textbf{35} (2007), no.~1, 180--205.

\bibitem{fang2005}
S.~Fang and T.~Zhang, A study of a class of stochastic differential equations
  with non-{L}ipschitzian coefficients, \emph{Probab. Theory Related Fields}
  \textbf{132} (2005), no.~3, 356--390.

\bibitem{fedrizzi2013}
E.~Fedrizzi and F.~Flandoli, H\"{o}lder flow and differentiability for {SDE}s
  with nonregular drift, \emph{Stoch. Anal. Appl.} \textbf{31} (2013), no.~4,
  708--736.

\bibitem{g2001}
I.~Gy\"{o}ngy and T.~Mart\'{\i}nez, On stochastic differential equations with
  locally unbounded drift, \emph{Czechoslovak Math. J.} \textbf{51(126)}
  (2001), no.~4, 763--783.

\bibitem{krylov1980}
N.~V. Krylov, \emph{Controlled diffusion processes}, volume~14 of
  \emph{Applications of Mathematics}, Springer-Verlag, New York-Berlin (1980),
  translated from the Russian by A. B. Aries.

\bibitem{krylov2008}
N.~V. Krylov, \emph{Lectures on Elliptic and Parabolic Equations in {{Sobolev}}
  Spaces}, volume~96 of \emph{Graduate Studies in Mathematics}, {American
  Mathematical Society, Providence, RI} (2008).

\bibitem{krylov2021d}
N.~V. Krylov, On diffusion processes with drift in {$L_d$}, \emph{Probab.
  Theory Related Fields} \textbf{179} (2021), no. 1-2, 165--199.

\bibitem{krylov2021b}
N.~V. Krylov, On stochastic equations with drift in {$L_d$}, \emph{Ann.
  Probab.} \textbf{49} (2021), no.~5, 2371--2398.

\bibitem{krylov2021a}
N.~V. Krylov, On stochastic {I}t\^{o} processes with drift in {$L_ d$},
  \emph{Stochastic Process. Appl.} \textbf{138} (2021), 1--25.

\bibitem{krylov2021c}
N.~V. Krylov, On strong solutions of {I}t\^{o}'s equations with {$\sigma\in
  W^1_d$} and {$b\in L_d$}, \emph{Ann. Probab.} \textbf{49} (2021), no.~6,
  3142--3167.

\bibitem{Krylov2005}
N.~V. Krylov and M.~R\"{o}ckner, Strong solutions of stochastic equations with
  singular time dependent drift, \emph{Probab. Theory Related Fields}
  \textbf{131} (2005), no.~2, 154--196.

\bibitem{kunita1990}
H.~Kunita, \emph{Stochastic flows and stochastic differential equations},
  volume~24 of \emph{Cambridge Studies in Advanced Mathematics}, Cambridge
  University Press, Cambridge (1990).

\bibitem{ladyzenskaja1968}
O.~A. Ladyzhenskaya, V.~A. Solonnikov and N.~N. Ural'tseva, \emph{Linear and
  quasilinear equations of parabolic type}, Translations of Mathematical
  Monographs, Vol. 23, American Mathematical Society, Providence, R.I. (1968),
  translated from the Russian by S. Smith.

\bibitem{li1994}
X.-M. Li, Strong {$p$}-completeness of stochastic differential equations and
  the existence of smooth flows on noncompact manifolds, \emph{Probab. Theory
  Related Fields} \textbf{100} (1994), no.~4, 485--511.

\bibitem{rockner2020}
M.~R{\"o}ckner and G.~Zhao, Sdes with critical time dependent drifts: weak
  solutions, \emph{arXiv preprint arXiv:2012.04161}  (2020).

\bibitem{rockner2021}
M.~R{\"o}ckner and G.~Zhao, Sdes with critical time dependent drifts: strong
  solutions, \emph{arXiv preprint arXiv:2103.05803}  (2021).

\bibitem{veretennikov1979}
A.~Y. Veretennikov, On the strong solutions of stochastic differential
  equations, \emph{Theory Probab. Appl.} \textbf{24} (1979), no.~2, 354--366.

\bibitem{wang2016}
F.-Y. Wang and X.~Zhang, Degenerate {SDE} with {H}\"{o}lder-{D}ini drift and
  non-{L}ipschitz noise coefficient, \emph{SIAM J. Math. Anal.} \textbf{48}
  (2016), no.~3, 2189--2226.

\bibitem{xie2016}
L.~Xie and X.~Zhang, Sobolev differentiable flows of {SDE}s with local
  {S}obolev and super-linear growth coefficients, \emph{Ann. Probab.}
  \textbf{44} (2016), no.~6, 3661--3687.

\bibitem{Yamada1981}
T.~Yamada and Y.~Ogura, On the strong comparison theorems for solutions of
  stochastic differential equations, \emph{Z. Wahrsch. Verw. Gebiete}
  \textbf{56} (1981), no.~1, 3--19.

\bibitem{zhang2005a}
X.~Zhang, Homeomorphic flows for multi-dimensional {SDE}s with non-{L}ipschitz
  coefficients, \emph{Stochastic Process. Appl.} \textbf{115} (2005), no.~3,
  435--448.

\bibitem{zhang2005}
X.~Zhang, Strong solutions of {SDES} with singular drift and {S}obolev
  diffusion coefficients, \emph{Stochastic Process. Appl.} \textbf{115} (2005),
  no.~11, 1805--1818.

\bibitem{zhang2011b}
X.~Zhang, Stochastic homeomorphism flows of {{SDEs}} with singular drifts and
  {{Sobolev}} diffusion coefficients, \emph{Electronic Journal of Probability}
  \textbf{16} (2011), no. 38, 1096--1116.

\bibitem{zhang2016}
X.~Zhang, Stochastic differential equations with {{Sobolev}} diffusion and
  singular drift and applications, \emph{Annals of Applied Probability}
  \textbf{26} (2016), no.~5, 2697--2732.

\bibitem{zhang2018}
X.~Zhang and G.~Zhao, Singular {{Brownian Diffusion Processes}},
  \emph{Communications in Mathematics and Statistics} \textbf{6} (2018), no.~4,
  533--581.

\bibitem{zvonkin1974}
A.~K. Zvonkin, A transformation of the phase space of a diffusion process that
  will remove the drift, \emph{Mat. Sb. (N.S.)} \textbf{93(135)} (1974),
  129--149, 152.

\end{thebibliography}
\end{document}